\documentclass[11pt]{article}
\usepackage[utf8]{inputenc}

\usepackage[nottoc]{tocbibind}

\usepackage[abbrev,lite,alphabetic]{amsrefs}
\usepackage[colorlinks,plainpages,urlcolor=blue,linkcolor=reference,citecolor=citation,colorlinks = true, hyperfootnotes=false]{hyperref}
\hypersetup{citecolor=gray, linkcolor = darkgray, urlcolor= gray}

\usepackage[T1]{fontenc}
\usepackage{etex}
\usepackage[shortlabels]{enumitem}
\usepackage{moreenum}
\usepackage{subfiles}
 \usepackage{lmodern}
\usepackage{titlesec}
\usepackage{wrapfig}
\usepackage{csquotes}
\usepackage{comment}

\usepackage{titletoc}

\usepackage{stmaryrd}

\usepackage{tikz-cd}
\usepackage{mathrsfs}
\usepackage{sseq}
\usepackage{mathdots}
\usepackage{thm-restate}
\usepackage{amsmath}
\usepackage{amsxtra}
\usepackage{amscd}
\usepackage{amsthm}
\usepackage{amsfonts}
\usepackage{amssymb}
\usepackage{mathtools}
\usepackage{slashed}
\usepackage[scr=rsfs]{mathalfa}
\usepackage{eucal}
\usepackage{bbm}
\usepackage[all,cmtip]{xy}

\usepackage{graphicx}
\usepackage{tikz}
\usepackage{morefloats}
\usepackage{pdfpages}
\usetikzlibrary{matrix,arrows,decorations.pathmorphing,quotes,angles,decorations.markings,shapes.misc}
\RequirePackage{color}

\definecolor{darkpastelpurple}{rgb}{0.59, 0.44, 0.84}
\usepackage{microtype}
\usepackage{epigraph}
\usepackage{tocloft}
\usepackage[colorinlistoftodos, textsize=tiny]{todonotes}

\usepackage[headsep=1cm,headheight=2cm,margin=2.7cm]{geometry}
\usepackage[framemethod=TikZ]{mdframed}
\usepackage{setspace}
\titleformat{\subsection}[hang]{\large\bfseries}{\thesubsection\hsp\textcolor{gray!75}{|}\hsp}{0pt}{\large\bfseries}
\titleformat{\subsubsection}[hang]{\bfseries}{\thesubsubsection\hsp\textcolor{gray!75}{|}\hsp}{0pt}{\bfseries}
\titleformat{\part}[display]{\Huge\bfseries}{\partname~\thepart:}{20pt}{}{}
\newcommand{\hsp}{\hspace{20pt}}
\titleformat{\section}[hang]{\Large\bfseries}{\thesection\hsp\textcolor{gray!75}{|}\hsp}{0pt}{\Large\bfseries}

\makeatletter
\let\pprevious@stem\@empty

\def\fooo#1#2\global\let\previous@stem\current@stem#3\zz{%
   \def#1{#2%
\global\let\pprevious@stem\previous@stem
\global\let\previous@stem\current@stem
#3}}

\def\foo#1{\expandafter\fooo\expandafter#1#1\zz}

\foo\generate@alphalabel

\let\generate@label\generate@alphalabel

\def\calc@alpha@suffix{%
    \@tempswafalse
    \compare@stems\previous@stem\current@stem
    \ifsame@stems
        \ifx\previous@year\current@year
            \@tempswatrue
        \fi
    \else
    \compare@stems\pprevious@stem\current@stem
    \ifsame@stems
        \ifx\previous@year\current@year
            \@tempswatrue
        \fi
    \fi
    \fi
    \if@tempswa
        \global\advance\alpha@suffix\@ne
        \edef\alpha@label@suffix{\@suffix@format\alpha@suffix}%
        \ifnum\alpha@suffix=\tw@
            \immediate\write\@auxout{%
                \string\ModifyBibLabel{\prev@citekey}%
            }%
        \fi
    \else
        \let\alpha@label@suffix\@empty
        \global\alpha@suffix\@ne
        \@xp\ifx \csname b@\current@citekey @suffix\endcsname \relax
        \else
            \edef\alpha@label@suffix{\@suffix@format\alpha@suffix}%
        \fi
    \fi
}

\makeatother

\newcounter{alphasect}
\def\alphainsection{0}

\let\oldsection=\section
\def\section{%
  \ifnum\alphainsection=1%
    \addtocounter{alphasect}{1}
  \fi%
\oldsection}%

\renewcommand\thesection{%
  \ifnum\alphainsection=1%
    \Alph{alphasect}%
  \else
    \arabic{section}%
  \fi%
}%

\usepackage{cleveref}

\definecolor{reference}{rgb}{0.20,0.36,0.74}
\definecolor{citation}{rgb}{0,.40,.80}

\usepackage{makeidx}

\makeindex

\usepackage{lipsum}
\usepackage{cleveref}
\crefname{section}{\S \!\!}{\S\S \!\!}
\crefname{equation}{}{}
\crefname{enumi}{}{}
\crefname{appendix}{\S \!\!}{\S\S \!\!}  

\setenumerate[1]{label={(\arabic*)}}

\allowdisplaybreaks

\theoremstyle{plain}
\newtheorem{theorem}{Theorem}[section]
\newtheorem{mainthm}{Theorem}

\newtheorem{proposition}[theorem]{Proposition}
\newtheorem{lemma}[theorem]{Lemma}
\newtheorem{corollary}[theorem]{Corollary}
\newtheorem*{corollary*}{Corollary}
\theoremstyle{definition}

\newtheorem{definition}[theorem]{Definition}

\newtheorem{notation}[theorem]{Notation}
\newtheorem{observation}[theorem]{Observation}
\newtheorem{remark}[theorem]{Remark}
\newtheorem{example}[theorem]{Example}
\newtheorem*{example*}{Example}
\newtheorem*{remark*}{Remark}
\newtheorem{proposal}[theorem]{Proposal}

\def\cA{\mathcal A}\def\cB{\mathcal B}\def\cC{\mathcal C}\def\cD{\mathcal D}
\def\cE{\mathcal E}
\def\cI{\mathcal I}\def\cL{\mathcal L}
\def\cM{\mathcal M}\def\cO{\mathcal O}
\def\cS{\mathcal S}
\def\cV{\mathcal V}\def\cW{\mathcal W}\def\cX{\mathcal X}

\newcommand{\fB}{\mathfrak{B}}
\newcommand{\fF}{\mathfrak{F}}
\newcommand{\fZ}{\mathfrak{Z}}
\newcommand{\fX}{\mathfrak{X}}

\newcommand{\Spaces}{{\cS}}

\newcommand{\ra}{\rightarrow}

\newcommand{\xra}{\xrightarrow}

\newcommand{\xla}{\xleftarrow}

\newcommand{\xhookra}{\xhookrightarrow}

\newcommand{\longla}{\longleftarrow}

\newcommand{\longra}{\longrightarrow}

\newcommand{\xlongra}[1]{\stackrel{#1}{\longra}}
\newcommand{\xlongla}[1]{\stackrel{#1}{\longla}}

\newcommand{\longhookra}{\lhook\joinrel\longrightarrow}
\newcommand{\longhookla}{\longleftarrow\joinrel\rhook}

\newcommand{\xlonghookla}[1]{\stackrel{#1}{\longhookla}}

\makeatletter
\providecommand{\leftsquigarrow}{%
  \mathrel{\mathpalette\reflect@squig\relax}%
}
\newcommand{\reflect@squig}[2]{%
  \reflectbox{$\m@th#1\rightsquigarrow$}%
}
\makeatother

\newcommand{\hookra}{\hookrightarrow}

\newcommand{\adjarr}{\rightleftarrows}

\newcommand{\Image}{{\sf Im \,}}

\newcommand{\adj}{\dashv}

\DeclareMathOperator{\uno}{\mathbbm{1}}
\newcommand{\ev}{{\sf ev}}

\newcommand{\Fun}{{\sf Fun}}
\newcommand{\Hom}{{\sf Hom}}

\newcommand{\Cat}{{\sf Cat}}

\renewcommand{\Pr}{{\sf Pr}}

\newcommand{\st}{{\sf st}}

\newcommand{\cotangent}{{\rm T}^*}

\newcommand{\PrL}{\Pr^L}

\newcommand{\PrLSt}{\Pr^{L,{\sf st}_{\kk}}}

\newcommand{\Mnd}{{\sf Mnd}}

\newcommand{\Kl}{{\sf Kl}}
\newcommand{\EM}{{\sf EM}}

\newcommand{\St}{{\sf St}}
\newcommand{\Mod}{{\sf Mod}}

\newcommand{\Perf}{{\sf Perf}}

\newcommand{\coCAlg}{{\sf coCAlg}}
\newcommand{\coMod}{{\sf coMod}}

\newcommand{\BiMod}{{\sf BiMod}}

\newcommand{\fgt}{{\sf fgt}}

\renewcommand{\lim}{{\sf lim}}
\newcommand{\lax}{{\sf lax}}

\newcommand{\const}{{\sf const}}
\newcommand{\pt}{{\sf pt}}

\newcommand{\id}{{\sf id}}
\newcommand{\Ar}{{\sf Ar}}
\newcommand{\Exact}{{\sf Exact}}

\newcommand{\LDK}{{\sf LDK}}

\newcommand{\bit}[1]{\textbf{\textit{#1}}}

\newcommand{\unital}{{\sf unital}}

\newcommand{\op}{{\sf op}}

\newcommand{\CMon}{{\sf CMon}}
\newcommand{\Mon}{{\sf Mon}}

\newcommand{\coBar}{{\sf coBar}}

\newcommand{\mon}{{\sf mon}}

\newcommand{\ex}{{\sf ex}}

\newcommand{\ff}{{\sf ff}}

\newcommand{\CAlg}{{\sf CAlg}}
\newcommand{\Alg}{{\sf Alg}}
\newcommand{\Idl}{{\sf Idl}}

\renewcommand{\hom}{{\sf hom}}

\newcommand{\what}{\widehat}

\renewcommand{\max}{{\sf max}}
\renewcommand{\min}{{\sf min}}

\newcommand{\colim}{{\sf colim}}
\newcommand{\cofib}{{\sf cofib}}
\newcommand{\fib}{{\sf fib}}

\renewcommand{\AA}{\mathbb{A}}

\newcommand{\CC}{\mathbb{C}}

\newcommand{\EE}{\mathbb{E}}
\newcommand{\GG}{\mathbb{G}}

\newcommand{\LL}{\mathbb{L}}

\renewcommand{\SS}{\mathbb{S}}

\newcommand{\ZZ}{\mathbb{Z}}

\newcommand{\kk}{\Bbbk}

\newcommand{\gr}{{\sf gr}}
\newcommand{\triv}{{\sf triv}}

\newcommand{\Env}{{\sf Env}}

\newcommand{\Aut}{\mathsf{aut}}
\newcommand{\End}{\mathsf{end}}

\newcommand{\Coh}{\mathsf{Coh}}

\newcommand{\QCoh}{\mathsf{QCoh}}

\newcommand{\IndCoh}{\mathsf{IndCoh}}

\newcommand{\lcone}{\triangleleft}

\newcommand{\rcone}{\triangleright}

\newcommand{\Sph}{\mathsf{Sph}}

\newcommand{\Adj}{\mathsf{Adj}}

\newcommand{\radjt}{{\sf r.adjt}}

\newcommand{\Tot}{{\sf Tot}}

\newcommand{\Perv}{\mathsf{Perv}}
\newcommand{\PS}{\mathsf{Perv}^{(2)}}

\def\bDelta{\mathbf\Delta}

\newcommand{\ul}{\underline}

\newcommand{\Loc}{\mathsf{Loc}}

\renewcommand{\sec}{\S}

\newcommand{\Ind}{{\sf Ind}}

\newcommand{\heart}{\heartsuit}

\renewcommand{\cV}{{{\mathcal{V}_{\kk}}}}
\newcommand{\Stu}{{\sf St}}
\renewcommand{\St}{{{\sf St}_{\kk}}}
\newcommand{\mrk}{{\sf mark}}

\title{Perverse schobers and 3d mirror symmetry}
\author{Benjamin Gammage, Justin Hilburn, and Aaron Mazel-Gee}
\date{\vspace{-5ex}}

\begin{document}

\maketitle

\begin{abstract}

The proposed physical duality known as {3d mirror symmetry} relates the geometries of dual pairs of holomorphic symplectic stacks. It has served in recent years as
a guiding principle for developments in representation theory. However, due to the lack of definitions, thus far only small pieces of the subject have been mathematically accessible. In this paper, we formulate abelian 3d mirror symmetry as an equivalence between a pair of 2-categories constructed from the algebraic and symplectic geometry, respectively, of Gale dual toric cotangent stacks.

In the simplest case, our theorem provides a spectral description of the 2-category of spherical functors -- i.e., perverse schobers on the affine line with singularities at the origin. We expect that our results can be extended from toric cotangent stacks to hypertoric varieties, which would provide a categorification of previous results on Koszul duality for hypertoric categories $\mathcal{O}$. Our methods also suggest approaches to 2-categorical 3d mirror symmetry for more general classes of spaces of interest in geometric representation theory.

Along the way, we establish two results that may be of independent interest: (1) a version of the theory of Smith ideals in the setting of stable $\infty$-categories; and (2) an ambidexterity result for co/limits of presentable enriched $\infty$-categories over $\infty$-groupoids.
\end{abstract}


\setcounter{tocdepth}{2}
\tableofcontents
 
\setcounter{section}{-1}

\section{Introduction}\label{sec:intro}

Mirror symmetry, which began life as a duality of 2-dimensional quantum field theories, entered mathematics in \cite{COGP} as a predicted equality between Gromov--Witten invariants and period integrals 
associated to a dual pair of Calabi--Yau manifolds. 
In \cite{K94}, this equality of numbers was upgraded to a proposed equivalence of categories.
These categories model the boundary conditions in a pair of topological field theories --
the A-model and B-model -- determined by the data of a symplectic or algebraic manifold.
Kontsevich's proposal prefigured the result \cites{Baez-Dolan,Costello-tcft,Lurie-cobordism} that an $n$-dimensional topological field theory is determined by its $(n-1)$-category of boundary conditions.

The program of \bit{3d mirror symmetry} \cite{IS96} (also known as \bit{symplectic duality} \cite{BLPW16}), which entails a duality between $3$-dimensional quantum field theories, has so far been understood mathematically as expressing a relation between invariants 
associated to dual pairs of symplectic resolutions.\footnote{The literature on this subject is vast; some representative research directions are described in \cites{BFN1,BDGH,Okounkov-ICM,Kamnitzer-survey}, and deeper references to the field may be found there.}
As in (2d) mirror symmetry,
most of these invariants probe topological field theories -- the 3d A- and B-models.\footnote{The 3d B-model is also known as Rozansky--Witten theory \cite{RW97}. The $3$d A-model 
studied in this paper is essentially $3$d generalized
Seiberg-Witten theory. In general, this latter theory is different than that considered in \cites{Kap-Vyas, Kap-Set-Vyas}.}
Therefore the deepest formulation of (topologically twisted) 3d mirror symmetry is as a conjectural equivalence of 2-categories.

Unfortunately, the mathematical
existence of these 2-categories is also conjectural.
The B-type 2-category, first studied in \cites{KRS,KR}, is expected to admit a formulation in the emerging mathematical language of coherent sheaves of categories \cites{Ga-shcat,Ari-talk}, but the A-type 2-category, an expected categorification of the Fukaya category,
is much more mysterious. The inspirational ICM address \cite{Tel-ICM} first formulated the A-type 2-category associated to pure gauge theory and described the equivalence of 2-categories predicted by 3d mirror symmetry. 

In this paper, we take the first steps beyond pure gauge theory, by proposing 2-categories associated to abelian gauge theories with matter.
In mathematical terms, we associate 2-categories to \bit{toric cotangent stacks} -- 
quotients of the form ${\cotangent}(\CC^n/G)$ for $G\subset (\CC^\times)^n$ a torus -- and prove equivalences between the 2-categories associated to dual pairs of stacks.
These equivalences, which can be understood in purely mathematical terms, reveal heretofore hidden structure inside the world of stable higher category theory.

In \Cref{subsec:intro-math}, we provide a purely mathematical account of our results, focusing on the basic duality between ${\cotangent}\CC$ and ${\cotangent}(\CC/\CC^\times).$ In this case, we understand 3d mirror symmetry as an equivalence between a 2-category of perverse schobers on $\CC$ and a 2-category of coherent sheaves of categories on $\CC/\CC^\times.$ Our main result is \Cref{mainthm:sph-spectral}, which gives a precise statement of this equivalence.
In \Cref{subsec:abelian-intro}, which may be skipped by a reader interested only in the mathematical results of this paper, we situate our results in the context of abelian gauge theories with matter and give a preview of our future applications to hypertoric categories $\cO.$

\subsection{Mathematical overview}\label{subsec:intro-math}

The notion of \bit{spherical functor} \cites{Anno-Logvinenko} has proven remarkably useful as an organizing principle in contexts throughout algebraic and symplectic geometry. 
\begin{definition}
An adjunction\footnote{The indices $\Phi,\Psi$ on the categories participating in this adjunction are meant to recall the notation for the nearby and vanishing cycles of a perverse sheaf.}
$S:\cC_\Phi\rightleftarrows\cC_\Psi:S^R$ between stable categories\footnote{Throughout this paper, we take the ``implicit $\infty$'' convention (see \Cref{subsection:notn.and.conventions}).} is a \bit{spherical adjunction} if it satisfies the requirement that the endofunctors $T_\Phi:=\fib(\id_{\cC_\Phi}\xra{\eta}S^RS)$ and $T_\Psi:=\cofib(SS^R\xra{\varepsilon}\id_{\cC_\Psi})$ are invertible. The left adjoint $S$ is the underlying \bit{spherical functor} of the spherical adjunction.\footnote{We prefer the term ``spherical adjunction'' over the more common term ``spherical functor'', since it is more symmetric and in particular adheres more closely to the intuition coming from symplectic topology. On the other hand, the terminology ``spherical functor'' will also be useful for us.}
\end{definition}

\noindent Examples include the pushforward of coherent sheaves on a divisor in algebraic geometry, or the ``Orlov/cup'' functor in symplectic geometry; examples coming from representation theory are also discussed in \cite{KS-schobers}*{\sec 4} and
\cite{KS-parabolic}*{\sec 8}.
It was recognized in \cite{KS-schobers} that spherical functors categorify the quiver presentation of perverse sheaves on $\CC$ with stratification $\cS_{toric}$ given by $\CC = 0\sqcup \CC^\times.$
\begin{theorem}[\cite{Verdier}]
The category $\Perv(\CC,\cS_{toric})$ is equivalent to the category of diagrams of vector spaces $u : V_\Phi \rightleftarrows V_\Psi : v$ satisfying the requirement that the endomorphisms $T_\Phi:= 1_\Phi-vu$ and $T_\Psi:=1_\Psi-uv$ are invertible.
\end{theorem}
\noindent For this reason, spherical functors are also referred to as \bit{perverse schobers} on $\CC$ with stratification $\cS_{toric}$ \cite{BKS}.  

%
The collection of all spherical functors (not fixing the source or target category) may be usefully organized into a 2-category $\Sph$, a full sub-2-category of the 2-category $\Fun(\Adj,\St)$ of stable adjunctions. As we shall see, this 2-category is remarkably rich: the simple step of demanding that the twists be invertible imposes a great deal of structure, including many symmetries not present for arbitrary stable adjunctions.




The mathematical content of our paper is a spectral description of the 2-category of spherical functors. To give a concrete explanation of what we mean by a ``spectral'' description, we consider the main precursor to our work, the equivalences of 2-categories discussed
in \cite{Tel-ICM}.

\begin{definition}
Let $X$ be a stack which is 1-affine in the sense of \cite{Ga-shcat}. We may associate to $X$ the 2-categories $\QCoh^{(2)}(X):= \Mod_{\Perf(X)}(\St)$ of \bit{quasicoherent sheaves of categories} and $\Loc^{(2)}(X):=\Fun(X_{\sf{B}},\St)$ of \bit{local systems of categories}, where we write $X_{\sf{B}}$ for the homotopy type (``Betti stack'') underlying $X$.
\end{definition}

\begin{example} \label{ex:intro-cat-fourier}
Fix a torus $T := (\CC^\times)^n$, and let $T^\vee$ be its Langlands dual. The (\textit{categorified}) \textit{Fourier transform} \cite{Tel-ICM}*{Remark 4.6} is the equivalence of 2-categories
\[
\QCoh^{(2)}(BT)
\simeq
\Loc^{(2)}(T^\vee)
~,
\]
which follows from the observation that $\QCoh^{(2)}(BT)$ and $\Loc^{(2)}(T^\vee)$ are both equivalent to the $2$-category of categories equipped with $n$ commuting automorphisms.

There is also a dual Fourier transform \cite{Tel-ICM}*{Proposition 4.1}
\[
\Loc^{(2)}(BT)
\simeq
\QCoh^{(2)}(T^\vee)
~,
\]
where now the classifying stack appears on the automorphic side instead of the spectral side. To see this, note that $\Loc^{(2)}(BT)$ is equivalent to the 2-category of categories $\cC$ equipped with a topological action of the torus $T$. Such an action is equivalent to the data of an $\EE_2$ map $\CC[T^\vee] \cong C_*(\Omega T) \ra HH^*(\cC),$ as explained in \cite{Tel-ICM}*{Theorem 2.5}.
\end{example}

\Cref{ex:intro-cat-fourier} expressed a duality between the spaces $BT$ and $T^\vee$. Similarly, our results will relate the space $\CC,$ equipped with stratification $0\sqcup \CC^\times,$ to a dual geometry given by the space $\CC/\CC^\times$ equipped with stratification $0/\CC^\times \sqcup \CC^\times / \CC^\times$. (In \Cref{subsec:abelian-intro}, we will give context for the geometric duality relating both these stratified spaces and the spaces described in \Cref{ex:intro-cat-fourier}.)

Consider the functor
\[
\Coh(\CC/\CC^{\times}) \xlongra{i^*} \Coh(B \CC^{\times})
\]
of pullback along the inclusion $i:0/\CC^\times\hookrightarrow \CC/\CC^\times.$ This functor
is spherical, and we prove that it is the {\em free} spherical functor on a categorified vanishing cycle, in the following sense.

\begin{mainthm}[\Cref{thm.gr.coreps.LHS}]
\label{mainthm}
The functor $\Phi:\Sph\to \St$ taking a spherical adjunction $\cC_\Phi \rightleftarrows\cC_\Psi$ to the category $\cC_\Phi$ is corepresented by the spherical functor $\Coh(\CC/\CC^{\times}) \xra{i^*} \Coh(B \CC^{\times}).$
\end{mainthm}

As a spherical functor, $i^*$ admits both adjoints $i_! \adj i^* \adj i_*,$ which are themselves spherical functors. Moreover, these adjoints $i_!$ and $i_*$ are in fact equivalent (when considered as objects of $\Sph$), and
corepresent the functor $\Psi$ taking a spherical adjunction $\cC_\Phi \rightleftarrows \cC_\Psi$ to the category $\cC_\Psi.$
%
%
%
%
Since the objects corepresenting the functors $\Phi$ and $\Psi$ jointly generate $\Sph,$ in the sense that the functor corepresented by their direct sum is conservative, in the spirit of the Barr-Beck theorem we may use their their monoidal endomorphism category to give a presentation of the full 2-category $\Sph$.
\begin{mainthm}[\Cref{prop:spectral-description}]\label{mainthm:sph-spectral}
The 2-category of spherical functors admits a canonical equivalence $\Sph \simeq \Mod_\cA(\St)$ with the 2-category of modules for the convolution monoidal category
\[
\cA:= \Coh\left(
(\CC/\CC^\times \sqcup 0/\CC^\times)
\times_{\CC/\CC^\times}
(\CC/\CC^\times \sqcup 0/\CC^\times)
\right).
\]
\end{mainthm}


As in \Cref{ex:intro-cat-fourier}, there is a dual form of \Cref{mainthm:sph-spectral} in which $\CC/\CC^\times = 0/\CC^\times \sqcup \CC^\times / \CC^\times$ appears on the automorphic side and $\CC = 0\sqcup \CC^\times$ appears on the spectral side.

\begin{mainthm}[\Cref{cor:ungauging-dim1}]\label{mainthm:s1sph-spectral}
The 2-category $\Sph$ admits a topological $S^1$-action, and the invariant 2-category admits a canonical equivalence $\Sph^{S^1} \simeq \Mod_{\cA'}(\St)$ with the 2-category of modules for the convolution monoidal category
\[
\cA' := \Coh\left(
(\CC \sqcup 0)
\times_{\CC}
(\CC \sqcup 0)
\right).
\]
\end{mainthm}
We deduce \Cref{mainthm:s1sph-spectral} from \Cref{mainthm:sph-spectral} by applying the following (twice categorified) Fourier transform.
\begin{mainthm}[\Cref{thm:4d-fourier}]\label{mainthm:4d-fourier}
For any torus $T$, there is an equivalence of $3$-categories
\[
\Loc^{(3)}(BT)
\simeq
\QCoh^{(3)}(BT^\vee)
~.
\]
This equivalence is suitably functorial; in particular, it intertwines the functor of $T$-invariants with the functor that forgets the $T^\vee$-action.
\end{mainthm}

Towards establishing the above theorems,
we prove two auxiliary results that may be of independent interest, which we now describe.

Our first auxiliary result is an $\infty$-categorical version of the theory of Smith ideals. In ordinary algebra, given a ring $R$, there is a bijective correspondence between ideals of $R$ and surjective ring homomorphisms out of $R$. Underlying this is the fact that a surjective homomorphism of abelian groups is likewise determined by its kernel. By contrast, in higher algebra, \textit{every} homomorphism of derived abelian groups is determined by its fiber. Thus, one may expect that every homomorphism of derived rings should be equivalent data to its fiber equipped with suitable additional structure; and this is what we establish.

\begin{mainthm}[\Cref{cor.ideals.are.algebras}]\label{mainthm:smith}
Fix a stably monoidal category $\cC$. For each algebra object $A \in \Alg(\cC)$, the fiber $\fib(\uno_\cC \xra{\eta} A) \in \cC$ admits a canonical enhancement to an \bit{ideal}, i.e.\! a monoidal functor $\ZZ_{\leq 0} \ra \cC$. Moreover, this construction defines an equivalence
\[
\Alg(\cC)
\xlongra{\sim}
\Idl(\cC)
:=
\Fun^\mon(\ZZ_{\leq 0},\cC)
~.
\]
\end{mainthm}
The theory of ideals developed in \Cref{mainthm:smith} is the main technical tool used in our understanding of spherical functors and the proof of \Cref{mainthm:sph-spectral}. As we explain in \Cref{rem:segal}, this theory is the necessary ingredient to correct and explain earlier work \cite{Segal-autoequivalences} on spherical functors.

Our second auxiliary result is an ambidexterity theorem for limits and colimits of presentable enriched categories.

\begin{mainthm}[\Cref{thm.prbl.ambidext}]
Fix a presentably symmetric monoidal category $\cW \in \CAlg(\Pr^L)$. For any $\infty$-groupoid $\cA$ and any functor $\cA \xra{F} \Pr^L_\cW$ to the category of presentably $\cW$-enriched categories, there is a canonical equivalence
\[
\colim_\cA(F)
\simeq
\lim_\cA(F)
~.
\]
\end{mainthm}

\subsection{Abelian gauge theories and hypertoric category $\mathcal{O}$}\label{subsec:abelian-intro}
It is expected that a symplectic stack $\fX$ admitting a $\CC^\times$ action for which the symplectic form has weight 2
determines a pair of 3-dimensional topological field theories, the 3d A-model and B-model, respectively, relating to the symplectic and algebraic geometry of $\fX.$ This situation encompasses all cotangent stacks $\fX={\cotangent}(M/G),$ for $M$ a smooth variety with action of a reductive group $G$, and all examples we consider will be of this form. Moreover, to a choice of Lagrangian $\LL\subset \fX,$ each of these theories should associate a 2-category of boundary conditions supported on $\LL.$

The \bit{3d mirror symmetry conjecture} suggests that a 3d A-model associated to some geometry $\fX$ may also be realized as a 3d B-model, and vice versa. More precisely, in good cases, given a stack $\fX,$ we can find a ``3d dual'' space $\fX^\vee$ such that the A-model with target $\fX$ is equivalent to the B-model with target $\fX^\vee.$\footnote{Note that not every 3d A-model and B-model arises in this way; thus, even a na\"ive version of the 3d mirror symmetry conjecture will allow that the theory dual to a 3d A-model or B-model associated to $\fX$ may not arise from the data of a symplectic stack $\fX^\vee.$ This is analogous to the phenomenon, familiar from 2d mirror symmetry, that a non-Calabi--Yau K\"ahler manifold will in general be mirror to a Landau--Ginzburg model (rather than another K\"ahler manifold).}
(When $\fX = {\cotangent}(N/G)$ for a linear $G$-representation $N$, the ring of functions on the dual space is produced by the Coulomb branch construction \cites{N,BFN1}.)
In this case, we expect moreover that given a Lagrangian $\LL\subset \fX,$ there ought to exist a Lagrangian $\LL^\vee \subset \fX^\vee$ such that there is an equivalence between the respective 2-categories of boundary conditions supported on $\LL$ and $\LL^\vee.$

\begin{example}\label{ex:3d-triv}
Let $\fX={\cotangent}(\pt/T)$ for some $n$-torus $T\simeq(\CC^\times)^n.$ The dual space in this case is the variety $\fX^\vee = {\cotangent}(T^\vee)$ given by the cotangent bundle to the dual torus. Each of these spaces has a natural Lagrangian given by the zero-section, $\LL=\pt/T\subset \fX$ and $\LL^\vee=T^\vee\subset \fX^\vee.$

To the geometric data $(\LL\subset \fX)$ and $(\LL^\vee\subset \fX^\vee)$, the 3d A-model assigns
the 2-categories $\Loc^{(2)}(\pt/T)$ and $\Loc^{(2)}(T^\vee),$ respectively, whereas the 3d B-model 
assigns $\QCoh^{(2)}(\pt/T)$ and $\QCoh^{(2)}(T^\vee).$ The equivalences of 2-categories predicted by 3d mirror symmetry are described in \Cref{ex:intro-cat-fourier}.
\end{example}

The results of this paper take place in the situation where the space $\fX$ is of the form ${\cotangent}(N/G),$ where $N$ is a vector space with linear action of the torus $G.$
Such a stack may be specified by the data of an exact sequence of tori
\begin{equation}\label{eq:toric-sequence}
\begin{tikzcd}
1 \arrow[r] &
G \arrow[r, "i"]&
(\CC^\times)^n \arrow[r, "p"]&
F \arrow[r] &
1,
\end{tikzcd}
\end{equation}
where we take $N=\CC^n$ equipped with the action of $G$ induced from the inclusion $i.$ 
\begin{definition}
The \bit{toric cotangent stack} determined by the exact sequence \Cref{eq:toric-sequence} is
\[
\fX(N,G) := {\cotangent}(N/G).
\]
\end{definition}

If $\fX = \fX(N,G)$ is specified as above, then the 3d dual $\fX^\vee$ is known \cites{dBHOO,BLPW10,BFN1}
to be a space of the same form; namely, it is the toric cotangent stack associated to the dual exact sequence of tori
\begin{equation}\label{eq:toric-dual-sequence}
\begin{tikzcd}
1 \arrow[r] &
F^\vee \arrow[r, "p^\vee"]&
(\CC^\times)^n \arrow[r, "i^\vee"]&
G^\vee \arrow[r] &
1.
\end{tikzcd}
\end{equation}
In other words, we identify the 3d dual as $\fX(N,G)^\vee = \fX(N^\vee,F^\vee).$

Now, a toric cotangent stack $\fX(N,G)$ 
admits a distinguished Lagrangian subspace.

\begin{definition}
We write $\LL_{toric}\subset \fX(N,G)$ for the Lagrangian
\[
\LL_{toric} = \bigcup_\alpha \cotangent_{\overline{\cO}_\alpha}(N/G)
\]
given by the union of conormals to closures $\overline{\cO}_\alpha$ of $(\CC^\times)^n$-orbits $\cO_\alpha$ in $N/G.$ When we discuss the dual cotangent stack $\fX^\vee = \fX(N^\vee,F^\vee),$ we will write $\LL_{toric}^\vee$ for the analogously defined Lagrangian in $\fX^\vee$.
\end{definition}
\begin{example}
If $n=1$ and $G=\{1\},$ then the space $\fX(N,G) = {\cotangent}\CC$ with the Lagrangian $\LL_{toric}=\CC\cup \cotangent_0\CC$ whose irreducible components are the zero-section and the conormal to $0\in \CC.$ The dual space $\fX^\vee$ is the stack $\fX(N^\vee,F^\vee) = {\cotangent}(\CC/\CC^\times),$ equipped with Lagrangian $\LL^\vee_{toric} = \CC/\CC^\times \cup \cotangent_{0/\CC^\times}(\CC/\CC^\times).$
\end{example}

In general, it is not understood how to associate a 2-category of boundary conditions to the data of a symplectic stack $\fX$ and a Lagrangian $\LL\subset \fX$: the assignment described in \Cref{ex:3d-triv} is one of the only cases where it has been possible to formulate 3d mirror symmetry as an equivalence of 2-categories. The main difficulty is in understanding the symplectic (A-type) 2-category; the B-type 2-category has been explored in several works \cite{KRS,KR,Ari-talk}.

In this paper, we propose the following definitions for the 2-categories of boundary conditions associated to the data of a toric cotangent stack $\fX(N,G)$
equipped with the Lagrangian $\LL_{toric}\subset \fX(N,G).$

\begin{proposal}\label{def:proposal}
    We model boundary conditions for the 3d A-model with target ${\cotangent}(N/G)$ supported on $\LL_{toric}$ by the 2-category 
    \[\PS(N/G,\cS_{toric}) := \PS(N,\cS_{toric})^G \] 
    of $G$-invariant perverse schobers on $N$ equipped with stratification $\cS_{toric}$ by $(\CC^\times)^n$-orbits (incarnated as $\Sph_n^G$ in \Cref{subsec:hypertoric}).
    We model boundary conditions for the 3d B-model with target ${\cotangent}(N/G)$ supported on $\LL_{toric}$ by the 2-category 
    \[\IndCoh^{(2)}(N/G,\cS_{toric})\] 
    of $G$-invariant Ind-coherent sheaves of categories with singular support in the union of conormals to strata in $\cS_{toric}.$ This latter 2-category may be defined \cite{Ari-talk} as the 2-category $\Mod_\cA(\St)$ of modules for the monoidal category of coherent sheaves
    \[
    \cA:=\Coh\left( \left( \bigsqcup_\alpha\overline{\cO}_\alpha \right) \times_{N/G} \left( \bigsqcup_\alpha\overline{\cO}_\alpha \right) \right),
    \]
    where we write $\bigsqcup_\alpha \overline{\cO}_\alpha$ for the disjoint union of closures of $(\CC^\times)^n$-orbits in $N/G.$
\end{proposal}

\begin{remark}
Physically, the $\hom$-categories in $3$d A- and B-models should be Fukaya--Seidel and matrix factorization categories, respectively, on path spaces, as can be deduced by compactifying the 3d theory on an interval.\footnote{By instead compactifying on an $S^1,$ one obtains the categories of line operators studied in \cite{DGGH,HR, BN-voa}.} The resulting 2d theories are twists of the gauged Landau--Ginzburg model described in \cite{BDGH}*{Appendix A}, whose superpotential is a complexification of the equivariant symplectic action functional on path spaces \cites{Cieliebak,Frauenfelder}. On the B-side, the resulting matrix factorization categories were computed in \cite{KRS,KR} and shown to match a $(\ZZ/2$-graded) version of the $\hom$-categories in $\IndCoh^{(2)}.$
On the A-side, the ``algebra of the infrared'' \cite{GMW,KS-aoi,KSS} presents the resulting Fukaya--Seidel categories as categorifications of Floer groups.\footnote{More precisely, these are categorifications of the moment Floer groups of \cite{Frauenfelder}.} This justifies our proposal to formalize the 2-category of boundary conditions in the 3d A-model as a categorification of the Fukaya category. (For the relation between perverse sheaves and the Fukaya category, see \cite{GPS3,Jin}.)
\end{remark}

\begin{remark}
In \Cref{def:proposal}, we have expressed the 2-categories which 3d mirror symmetry associates to a cotangent bundle ${\cotangent}(N/G)$ in terms of the geometry of the base $N/G.$ We expect that there will eventually be more ``microlocal'' definitions of these 2-categories which do not privilege a particular polarization.
\end{remark}

The prediction that (topologically twisted) dual abelian 3d theories have equivalent boundary conditions may now be formulated precisely as an equivalence between the 2-categories described above. This is what we prove:

\begin{mainthm}[\Cref{thm:higherdim-mainthm}]\label{mainthm:boundary-duality}
Let $N=\CC^n$ be a vector space equipped with an action of torus $G\subset (\CC^\times)^n,$ and let $F=(\CC^\times)^n/G$ be the quotient torus as in \Cref{eq:toric-sequence}. Then
there is an equivalence
\[
\PS(N/G,\cS_{toric}) \simeq \IndCoh^{(2)}(N^\vee/F^\vee, \cS_{toric})
\]
between the A-model 2-category associated to $\LL_{toric}\subset \fX(N,G)$ and the B-model 2-category associated to $\LL_{toric}^\vee\subset \fX(N^\vee,F^\vee).$
\end{mainthm}

As a corollary of \Cref{mainthm:boundary-duality}, we are able to rederive the prescription, from \cite{BFN1}, for the algebra of functions on the 3d mirror dual to $\fX(N,G)$, directly from our 2-category of A-type boundary conditions. As in their approach, we follow the expectation that this algebra should be computed by endomorphisms of the trivial line defect inside the category of line operators. In our formulation, this is the endomorphism algebra of the unit object $\uno$ in the categorical center of the 2-category of boundary conditions. (By definition, for a 2-category of the form $\cC = \Mod_{\cA}(\cD),$ for some algebra $\cA\in\Alg(\cD),$ the categorical center $Z_{cat}(\cC)$ is defined to be the center $Z(\cA)\in\Alg_{\EE_2}(\cD)$ of the algebra $\cA.$)
\begin{mainthm}[\Cref{thm:e3-functions}]
There is an equivalence 
\[
\End_{Z_{cat}(\Perv^{(2)}(N/G,\cS_{toric})}(\uno)\simeq \cO(\fX(N^\vee,F^\vee))_{regraded}
\]
between the algebra of endomorphisms of the unit in the categorical center of the 2-category of A-type boundary conditions for $\fX(N,G)$ and a regraded version of the algebra of functions on $\fX(N^\vee,F^\vee)$ where functions linear in the cotangent direction are placed into homological degree 2.
\end{mainthm}

\begin{remark} \label{rem:3dand4d}
As we have mentioned previously, the proof of \Cref{mainthm:boundary-duality} exploits extra structure in these 3d theories: their residual symmetries enhance them to boundary conditions for 4-dimensional abelian gauge theories \cite{GW}. The equivalence described in \Cref{mainthm:4d-fourier} is a form of the Betti Geometric Langlands equivalence, relating the A- and B-type twists of these 4-dimensional theories. This idea was exploited in \cite{BFN-ring, BFGT, HR} to compute rings/categories of local/line operators and forms the underpinning for the relative Langlands program of Ben-Zvi, Sakellaridis, and Venkatesh.
\end{remark}

From the perspective of supersymmetric gauge theory or representation theory, the support Lagrangian $\LL_{toric}\subset \fX=\fX(N,G)$ may seem unnatural; usually, as in the theory of categories $\cO$ for symplectic resolutions \cite{BLPW16}, one studies a Lagrangian defined as the attracting set for some Hamiltonian $\CC^\times$-action on $\fX.$ Note that a toric cotangent stack $\fX(N,G)$ defined via the exact sequence \Cref{eq:toric-sequence} admits a Hamiltonian action by the residual torus $F$; therefore, a Hamiltonian $\CC^\times$-action may be determined from a cocharacter $m\in X_*(F)$ of the torus $F,$ specifying a 1-parameter subgroup acting on $\fX(N,G).$

\begin{definition}
Given a cocharacter $m\in X_*(F)$, the \bit{relative skeleton} $\LL^m\subset \fX(N,G)$ is obtained by Hamiltonian reduction from the Lagrangian $\overline{\LL}^m\subset {\cotangent}N$ which is the union of attracting sets for all cocharacters of $(\CC^\times)^n$ which map to $m$ under the projection $p_*$ \cite{Web-gencat-O}*{Definition 2.8}. Following \cite{GPS2}, we refer to the process of obtaining $\LL^m$ by removing components of $\LL_{toric}$ as \bit{stop removal}.
\end{definition}

\noindent Our choice of Lagrangian $\LL_{toric}$ is therefore not so unnatural: it is a sort of ``master Lagrangian'' containing the relative skeleta $\LL^m$ for all possible choices of parameter $m.$ 

At this point it is reasonable to ask how the passage from $\LL_{toric}$ to the sub-Lagrangian $\LL^m$ is reflected on the 3d mirror dual. We first recall the parameter $t$ which is dual to $m$.

\begin{definition}
Given a character $t\in X^*(G)$ for the torus $G$, we may treat $t$ as a GIT stability parameter to define the \bit{hypertoric orbifold} $\fX(N,G)_t:={\cotangent}N/\!\!/\!\!/_t G$ \cites{Goto,Biel-Dan,Haus-Sturm}, which embeds as an open substack in $\fX(N,G).$ We write $\LL_t:=\LL_{toric}\cap \fX(N,G)_t$ for the restriction of the toric Lagrangian $\LL_{toric}$ to this subset. We refer to the passage from $\fX(N,G)$ to $\fX(N,G)_t$ as \bit{microrestriction}.
\end{definition}

\begin{definition}
Given $m\in X_*(F)$ and $t\in X^*(G),$
we may simultaneously impose the GIT stability condition $t$ and pass to the relative skeleton $\LL_t^m:=\LL^m\cap \fX(G,N)_t.$ This is the \bit{category $\cO$ Lagrangian} associated to the exact sequence \Cref{eq:toric-sequence} and the parameters $t$ and $m.$
\end{definition}

Observe that when we dualize exact sequence \Cref{eq:toric-sequence} to \Cref{eq:toric-dual-sequence}, the roles of the parameters $t$ and $m$ are exchanged: $m$ becomes a character of the subtorus $F^\vee,$ while $t$ becomes a cocharacter of the quotient torus $G^\vee.$
The prediction of 3d mirror symmetry is that $\fX(N,G)_t$ equipped with Lagrangian $\LL_t^m$ is dual to $\fX(N^\vee,F^\vee)_m$ equipped with Lagrangian $(\LL^\vee)_m^t.$
In slogan form: ``Stop removal is dual to microrestriction.''

In the sequel to this paper, we will explain the effects of stop removal and microrestriction on both the 3d A-model and B-model 2-categories of boundary conditions. We will then be able to make the above slogan into a rigorous mathematical statement, through which we will establish 3d mirror symmetry as an equivalence between ``2-categories $\cO$'' associated to $\LL_t^m$ and $(\LL^\vee)_m^t.$

\subsection{Notation and conventions}
\label{subsection:notn.and.conventions}

Our categorical conventions, which are largely standard, are articulated in \Cref{section.appendix.categorical.conventions}. Here we highlight a few of those and lay out the other conventions that we will use.

We reemphasize here that we take the ``implicit $\infty$'' convention; we use the term ``discrete'' to refer to a non-homotopical object. 

We work $\kk$-linearly for $\kk$ a commutative ring or commutative ring spectrum.
We write $\what{\cV} := \Mod_\kk$ for its category of modules, and $\cV := \Perf_\kk \subseteq \Mod_\kk$ for the subcategory of its perfect modules. Hence, for $\kk$ an ordinary ring we write $\what{\cV}^\heart$ (resp.\! $\cV^\heart$) for its discrete category of discrete (resp.\! finitely-generated projective discrete) modules.

For $0 \leq n \leq 2$, we write $\Stu_{\kk,n}$ for the $(n+1)$-category of $\kk$-linear small stable $n$-categories (with $\Stu_{\kk,0} := \what{\cV}$, and often simply writing $\St := \Stu_{\kk,1}$), and we write
\[
\Cat_n \xra{\Sigma^{(\infty,n)}_+} \Stu_{\kk,n}
\]
for the left adjoint to the forgetful functor, so that we may understand $\Sigma^{(\infty,n)}_+$ as the ``$n$-categorical $\kk$-stabilization'' functor; see \Cref{subsec:stable-ncats} for more details.
We will occasionally wish to refer to the 2-category of \textit{presentable} stable $\kk$-linear categories, which we denote by $\PrLSt := \PrLSt_1 \simeq \Mod_{\what{\cV}}(\Pr^L)$.

We write $\ZZ$ for the poset of integers with its usual ordering, which is symmetric monoidal by addition. We write $\ZZ_{\leq 0} \subset \ZZ$ for the full symmetric monoidal subcategory on the nonpositive elements, and we write $\ZZ^\delta \subset \ZZ$ for its (symmetric monoidal) maximal subgroupoid. 

We write
\[
f\cV
:=
\Fun(\ZZ^\op,\what{\cV})^\omega
\simeq
\Sigma^{(\infty,1)}_+ \ZZ
\quad
\text{and}
\quad
g\cV
:=
\Fun((\ZZ^\delta)^\op,\what{\cV})^\omega
\simeq
\Sigma^{(\infty,1)}_+ (\ZZ^\delta)
~.
\]
As the notation suggests, these are to be understood as the stable categories of {\em filtered} and {\em graded} $\kk$-vector spaces, respectively: see \cite{Gwilliam-Pavlov} or \cite{Lurie-Rotation}*{\S 3} for a review.
Note the implicit finiteness conditions:
\begin{itemize}

\item A filtered $\kk$-module lies in $f\cV$ if and only if its values are perfect, it is eventually constant to the left (i.e.\! at large negative integers), and it is eventually constant at $0_\cV$ to the right (i.e.\! at large positive integers).

\item A graded $\kk$-module lies in $g\cV$ if and only if its values are perfect and it is eventually constant at $0_\cV$ both to the left and to the right.
\end{itemize}
Nevertheless, we simply refer to these as the categories of filtered and graded modules, respectively.

Given a category $\cA$ and an object $a \in \cA$, we write
\[
h_a
:=
\Sigma^{(\infty,0)}_+ \hom_\cA(-,a)
\in
\Sigma^{(\infty,1)}_+ \cA
\quad
\text{and}
\quad
h^a
:=
\Sigma^{(\infty,0)}_+ \hom_\cA(a,-)
\in
\Sigma^{(\infty,1)}_+ \cA^\op
\]
for its stabilized Yoneda functors. In particular, for any $n \in \ZZ$ we have the filtered $\kk$-module
\[
f\cV
\ni
h_n
:=
\Sigma^{(\infty,0)}_+ \hom_{\ZZ}(-,n)
=
\left(
i
\longmapsto
\left\{
\begin{array}{ll}
\kk~,
&
i \leq n
\\
0~,
&
i > n
\end{array}
\right.
\right)
\]
(with nontrivial structure maps $\kk \xra{\id_\kk} \kk$). On the other hand, in the case that $\cA = \ZZ^\delta$, we introduce the particular notation $\delta_n := h_n$ for the graded $\kk$-module
\[
g\cV
\ni
\delta_n
:=
\Sigma^{(\infty,0)}_+ \hom_{\ZZ^\delta}(-,n)
=\left(
i
\longmapsto
\left\{
\begin{array}{ll}
\kk~,
&
i=n
\\
0~,
&
i \not= n
\end{array}
\right.
\right)
~.
\]

We write $
\langle 1 \rangle$ for the shift automorphism of both $f\cV$ and $g\cV$, given in both cases by the formula $(V\langle 1\rangle)^n := V^{n-1}$. This convention is motivated by the formulas for the application of shift automorphism to standard objects:
\[
h_n \langle 1\rangle \cong h_n \otimes h_1 \cong h_{n+1} \in f\cV
\quad
\text{and}
\quad
\delta_n \langle 1\rangle \cong \delta_n \otimes \delta_1 \cong \delta_{n+1} \in g\cV.
\]


\subsection{Acknowledgments}

The authors are grateful to Dima Arinkin, David Ben-Zvi, Alexei Oblomkov, Lev Rozansky, and Constantin Teleman for helpful conversations, encouragement, and inspiration. JH would also like to thank Mathew Bullimore, Tudor Dimofte, and Davide Gaiotto for their patient explanations of the physics of $3$d $\mathcal{N}=4$ theories and German Stefanich for discussions about \Cref{mainthm:4d-fourier} and the 3-categorical Betti Langlands correspondence.

BG acknowledges the support of an NSF Postdoctoral Research Fellowship, DMS-2001897. 

JH is part of the Simons Collaboration on Homological Mirror Symmetry supported by Simons Grant 390287. This research was supported in part by Perimeter Institute for Theoretical Physics. Research at Perimeter Institute is supported by the Government of Canada through the Department of Innovation, Science and Economic Development Canada and by the Province of Ontario through the Ministry of Research, Innovation and Science.

AMG acknowledges the support of NSF grant DMS-2105031.

\setcounter{section}{0}

\section{Algebras and ideals}
\label{section.algs.and.ideals}

In this section, we study algebra objects in a stably monoidal category $\cC := (\cC,\boxtimes) \in \Alg(\St)$. Our main result here (\Cref{cor.ideals.are.algebras}) proves that an algebra object $A \in \Alg(\cC)$ is equivalent data to its corresponding \textit{ideal}: viz., the object $I := \fib(\uno \xra{\eta} A) \in \cC$, equipped with suitable multiplicative structure.

On the one hand, this is basically the theory of Smith ideals \cite{Hovey-Smithideals}, ported over from model categories to $\infty$-categories.\footnote{Note however that it is not a priori clear what exactly an $\infty$-categorical ideal should be -- a drawback of working in the model-categorical setting.} On the other hand, this can also be viewed as an enhancement of \cite[Proposition 2.14]{MNN-nilp}, which says that the Lurie--Dold--Kan equivalence
\[
\Fun(\bDelta,\cC)
\xlongra{\sim}
\Fun(\ZZ_{\leq 0},\cC)
\]
carries the cobar complex $\coBar(A) \in \Fun(\bDelta,\cC)$ to the object $\cofib(T_\bullet(A) \ra \const(\uno)) \in \Fun(\ZZ_{\leq 0} , \cC)$, where we write $T_\bullet(A)$ for the \textit{Adams tower} of $A$.\footnote{Note that \cite[Construction 2.2]{MNN-nilp} defines the Adams tower asymmetrically (denoted as $T_\bullet(A,\uno)$ there), privileging left over right. By contrast, ideals are required essentially by definition to have these maps be equivalent (see \cite[Proposition 1.7]{Hovey-Smithideals}). Such symmetry is implicitly present in our work as well.} Namely, we upgrade $\coBar(A)$ to the \textit{augmented} cobar complex $\coBar_+(A) \in \Fun(\bDelta_+,\cC)$, which we moreover equip with the data of \textit{monoidality}. This becomes equivalent data to the object $A \in \Alg(\cC)$, because $\bDelta_+$ is the free monoidal ($\infty$-)category containing an algebra object.

\subsection{Monoidal structures on arrow categories}

Because $\cC$ is a stable category, taking cofibers and fibers are inverse operations: they assemble into an equivalence
\begin{equation}
\label{equivce.by.cofib.and.fib}
\begin{tikzcd}[column sep=2cm]
\Ar(\cC)
\arrow[yshift=0.9ex]{r}{\cofib}
\arrow[leftarrow, yshift=-0.9ex]{r}[yshift=-0.2ex]{\sim}[swap]{\fib}
&
\Ar(\cC)~.
\end{tikzcd}
\end{equation}
Using the monoidal structure of $\cC$, we will upgrade the equivalence \Cref{equivce.by.cofib.and.fib} to a monoidal equivalence (\Cref{lem.equivce.of.monoidal.strs.on.Ar.C}) with respect to monoidal structures that we introduce presently.

\begin{notation}
Note that the category $[1] \in \Cat$ admits two distinct monoidal structures: minimum and maximum.\footnote{The minimum is also the categorical product (as well as the usual product of integers), while the maximum is also the categorical coproduct.} Choosing either such structure makes the arrow category $\Ar(\cC)=\Fun([1],\cC)$ into a category of functors between monoidal categories, which therefore has itself a monoidal structure given by the Day convolution product
(first introduced in \cite{Day-convolution} and generalized to the $\infty$-categorical context in \cite{Glasman} and \cite{Lurie-HA}*{\S 2.2.6}). 
We denote these respective monoidal structures on $\Ar(\cC)$ by
\[
(\Ar(\cC),\boxdot)
:=
\Fun ( ([1],\min) , (\cC ,\boxtimes) )
\quad
\text{and}
\quad
(\Ar(\cC),\boxtimes)
:=
\Fun ( ([1],\max) , (\cC,\boxtimes) )
~.
\]
\end{notation}

\begin{observation}
Given a monoidal category $\cI\in \Mon(\Cat)$, the Day convolution of functors $F_1,\ldots, F_n \in \Fun(\cI,\cC)$ (if it exists) is given by the left Kan extension
\begin{equation}
\label{LKE.defn.of.day.conv}
\begin{tikzcd}[column sep=2cm]
\cI^{\times n}
\arrow{r}{F_1 \times \cdots \times F_n}
\arrow{d}[swap]{\otimes_\cI}
&
\cE^{\times n}
\arrow{r}{\otimes_\cE}
&
\cE
\\
\cI
\arrow[dashed, bend right=15]{rru}[sloped, swap]{F_1 \circledast \cdots \circledast F_n}
\end{tikzcd}
\end{equation}
(see e.g.\! \cite{Lurie-HA}*{Example 2.2.6.17}). So, the Day convolution evaluates at an object $i \in \cI$ as
\begin{equation}
\label{colimit.formula.for.day.conv}
(F_1 \circledast \cdots \circledast F_n)(i) \simeq \colim_{(i_1\otimes \cdots \otimes i_n \to i) \in \cI^{\times n} \times_\cI \cI_{/i}} \left( F_1(i_1)\otimes \cdots \otimes F_n(i_n) \right)
~.
\end{equation}
In the present situation, the above monoidal structures on $\Ar(\cC) := \Fun([1],\cC)$ are therefore given by the formulas
\[
(A \longra B)
\boxdot
(A' \longra B')
\simeq
\left( \left( A \boxtimes B' \coprod_{A \boxtimes A'} B \boxtimes A' \right)
\longra
B \boxtimes B'
\right)
\]
and
\[
(A \longra B)
\boxtimes
(A' \longra B')
\simeq
( A \boxtimes A'
\longra B \boxtimes B')
~.
\]
That is, they are respectively the \textit{pushout product} monoidal structure and the \textit{pointwise} monoidal structure.\footnote{For the latter, observe that taking $\cI = ([1],\max)$, the categories indexing the colimits \Cref{colimit.formula.for.day.conv} admit terminal objects $(0 \otimes \cdots \otimes 0) \ra 0$ (for $i=0$) and $(1 \otimes \cdots \otimes 1) \ra 1$ (for $i=1$). Moreover, given a morphism $i \ra i'$ in $\cI$, the functoriality of the Day convolution $F_1 \circledast \cdots \circledast F_n$ is given by the functoriality of colimits with respect to the induced functor $\cI^{\times n} \times_\cI \cI_{/i} \ra \cI^{\times n} \times_\cI \cI_{/i'}$. Hence, in this case, the left Kan extension \Cref{LKE.defn.of.day.conv} is simply given by precomposition with the diagonal section.}
\end{observation}

\begin{lemma}
\label{lem.equivce.of.monoidal.strs.on.Ar.C}
The equivalence \Cref{equivce.by.cofib.and.fib} canonically upgrades to a monoidal equivalence
\[
\begin{tikzcd}[column sep=2cm]
(\Ar(\cC),\boxdot)
\arrow[yshift=0.9ex]{r}{\cofib}
\arrow[leftarrow, yshift=-0.9ex]{r}[yshift=-0.2ex]{\sim}[swap]{\fib}
&
(\Ar(\cC),\boxtimes)~.
\end{tikzcd}
\]
\end{lemma}

\begin{proof}
Consider the product monoidal structure $([1] \times [1],\min \times \max) := ([1],\min) \times ([1],\max)$. Day convolution yields a monoidal structure on the category $\Ar^2(\cC) := \Fun([1]^{\times 2},\cC)$, which is given by the formula
\begin{equation}
\label{day.conv.on.squares}
\left(
\begin{tikzcd}
C
\arrow{r}
&
D
\\
A
\arrow{u}
\arrow{r}
&
B
\arrow{u}
\end{tikzcd}
\right)
~,~
\left(
\begin{tikzcd}
C'
\arrow{r}
&
D'
\\
A'
\arrow{u}
\arrow{r}
&
B'
\arrow{u}
\end{tikzcd}
\right)
\longmapsto
\left(
\begin{tikzcd}
{\displaystyle
C \boxtimes D' \coprod_{C \boxtimes C'} D \boxtimes C'
}
\arrow{r}
&
D \boxtimes D'
\\
{\displaystyle
A \boxtimes B' \coprod_{A \boxtimes A'} B \boxtimes A'
}
\arrow{u}
\arrow{r}
&
B \boxtimes B'
\arrow{u}
\end{tikzcd}
\right)
~;
\end{equation}
its unit object is
\[
\uno_{\Ar^2(\cC)}
\simeq
\left(
\begin{tikzcd}
0
\arrow{r}
&
\uno
\\
0
\arrow{r}
\arrow{u}{\id}
&
\uno
\arrow{u}[swap]{\id}
\end{tikzcd}
\right)
~.
\]

Let us write $\Exact(\cC) \subseteq \Ar^2(\cC)$ for the subcategory on the co/fiber sequences in $\cC$ -- i.e., the functors that select exact squares and that moreover carry the object $(0,1) \in [1]^{\times 2}$ to the zero object $0 \in \cC$. We claim that this is a monoidal subcategory.\footnote{This argument is essentially contained in the proof of \cite[Theorem 1.4]{Hovey-Smithideals}.} First of all, it clearly contains the unit object. To show that it is stable under the monoidal structure, choose any pair of objects $(A \xra{\varphi} B),(A' \xra{\varphi'} B') \in \Ar(\cC)$ and consider the diagram
\begin{equation}
\label{three.by.three.grid.for.proving.cofib.takes.pushout.prod.to.ptwise}
\begin{tikzcd}[row sep=1.5cm, column sep=1.5cm]
0
&
A \boxtimes B'
\arrow{r}{\varphi \boxtimes \id}
\arrow{l}
&
B \boxtimes B'
\\
0
\arrow{u}{\id}
\arrow{d}[swap]{\id}
&
A \boxtimes A'
\arrow{r}{\varphi \boxtimes \id}
\arrow{l}
\arrow{u}{\id \boxtimes \varphi'}
\arrow{d}[swap]{\varphi \boxtimes \id}
&
B \boxtimes A'
\arrow{u}[swap]{\id \boxtimes \varphi'}
\arrow{d}{\id}
\\
0
&
B \boxtimes A'
\arrow{r}[swap]{\id}
\arrow{l}
&
B \boxtimes A'
\end{tikzcd}
\end{equation}
in $\cC$ (a span of spans). Because left Kan extensions compose, we can compute its colimit either by taking the pushout of the vertical pushouts or by taking the pushout of the horizontal pushouts. Hence, we compute on the one hand that
\[
\colim\left(\Cref{three.by.three.grid.for.proving.cofib.takes.pushout.prod.to.ptwise}\right)
\simeq
\colim
\left(
\begin{tikzcd}
0
&
{\displaystyle
A \boxtimes B' \coprod_{A \boxtimes A'} B \boxtimes A'
}
\arrow{l}
\arrow{r}{\varphi \boxdot \varphi'}
&
B \boxtimes B'
\end{tikzcd}
\right)
=:
\cofib(\varphi \boxdot \varphi')
~,
\]
and on the other hand that
\[
\colim\left(\Cref{three.by.three.grid.for.proving.cofib.takes.pushout.prod.to.ptwise}\right)
\simeq
\colim
\left(
\begin{tikzcd}
\cofib(\varphi) \boxtimes B'
\\
\cofib(\varphi) \boxtimes A'
\arrow{u}{\id \boxtimes \varphi'}
\arrow{d}
\\
0
\end{tikzcd}
\right)
\simeq
\cofib(\varphi) \boxtimes \cofib(\varphi')
~.
\]
It follows that the formula \Cref{day.conv.on.squares} indeed preserves the subcategory $\Exact(\cC) \subseteq \Ar^2(\cC)$ (setting $C=C'=0$, $D = \cofib(A \xra{\varphi} B)$, and $D' = \cofib(A' \xra{\varphi'} B')$).

To conclude the proof, observe the monoidal functors
\[
([1],\min)
\xra{(\id , \const_0 )}
([1],\min) \times ([1],\max)
\xla{(\const_1,\id)}
([1],\max)
~,
\]
which by restriction induce monoidal functors
\[
(\Ar(\cC),\boxdot)
\longla
\Ar^2(\cC)
\longra
(\Ar(\cC),\boxtimes)
~.
\]
Because $\cC$ is stable, these restrict to monoidal equivalences
\[
(\Ar(\cC),\boxdot)
\xlongla{\sim}
\Exact(\cC)
\xlongra{\sim}
(\Ar(\cC),\boxtimes)
~,
\]
which by construction compose to the inverse equivalences  \Cref{equivce.by.cofib.and.fib} on underlying categories.
\end{proof}

\subsection{Algebras and ideals}

\begin{definition}\label{def.smith.ideal}
An \bit{ideal} in (the unit object of) $\cC$ is a monoidal functor $\ZZ_{\leq 0} \xra{I_\bullet} \cC$.\footnote{More generally (but actually not), an ideal in an algebra $A \in \Alg(\cC)$ is a monoidal functor $\ZZ_{\leq 0} \ra \BiMod_{(A,A)}(\cC)$.} We simply write $I := I_{-1} \in \cC$ and refer to this as the \textit{underlying object} of the ideal; using this notation, we have canonical equivalences $I_{n} \simeq I^{\boxtimes (-n)}$ for all $n \in \ZZ_{\leq 0}$. We write $\Idl(\cC) := \Fun^\mon ( \ZZ_{\leq 0} , \cC)$ for the category of ideals in $\cC$.
\end{definition}

\begin{proposition}
\label{cor.ideals.are.algebras}
There are canonical equivalences
\[
\begin{tikzcd}[column sep=3cm]
\Idl(\cC)
\arrow[yshift=0.9ex]{r}{I_\bullet \longmapsto \cofib(I \longra \uno)}
\arrow[leftarrow, yshift=-0.9ex]{r}[yshift=-0.2ex]{\sim}[swap]{\fib(\uno \longra A) \longmapsfrom A}
&
\Alg(\cC)~.
\end{tikzcd}
\]
\end{proposition}

\begin{proof}
Observing the equivalences
\[
\Idl(\cC)
\simeq
\Fun^{\ex,\mon}(\Sigma^{(\infty,1)}_+(\ZZ_{\leq 0}),\cC)
\quad
\text{and}
\quad
\Alg(\cC)
\simeq
\Fun^{\ex,\mon}(\Sigma^{(\infty,1)}_+(\bDelta_+),\cC)
~,
\]
we see that it is (necessary and) sufficient to construct an equivalence $\Sigma^{(\infty,1)}_+(\ZZ_{\leq 0}) \simeq \Sigma^{(\infty,1)}_+(\bDelta_+)$ in $\Alg(\St)$ that behaves as described. For this, we first construct a monoidal functor
\[
\bDelta_+
\xlongra{\varphi}
\Sigma^{(\infty,1)}_+(\ZZ_{\leq 0})
~,
\]
or equivalently an object of $\Alg(\Sigma^{(\infty,1)}_+(\ZZ_{\leq 0}),\otimes)$, where we write $\otimes$ for the monoidal structure of $\Sigma^{(\infty,1)}_+(\ZZ_{\leq 0})$ (inherited from the monoidal structure $+$ of $\ZZ_{\leq 0}$ by Day convolution). In turn, for this we begin with the evident laxly monoidal strictly unital functor
\[
([1],\min)
\xlongra{\varphi'}
(\ZZ_{\leq 0},+)
\]
between monoidal (discrete) categories, whose underlying functor classifies the morphism $-1 \ra 0$; its nontrivial laxness data consists of the morphism
\[
\varphi'(0)
+
\varphi'(0)
=
-2
\longra
-1
=
\varphi'(\min(0,0))
~.
\]
We then define $\varphi$ to be the image of $\varphi'$ under the composite left vertical functor in the commutative diagram
\[ \begin{tikzcd}
\Fun^{\lax.\mon,\unital}(([1],\min),(\ZZ_{\leq 0},+))
\arrow[hook]{r}{\ff}
\arrow{d}
&
\Fun^{\lax.\mon}(([1],\min),(\ZZ_{\leq 0},+))
\arrow{d}
\\
\Fun^{\lax.\mon,\unital}(([1],\min),(\Sigma^{(\infty,1)}_+(\ZZ_{\leq 0}),\otimes))
\arrow[hook]{r}{\ff}
\arrow[leftrightarrow]{d}[sloped, anchor=north]{\sim}
&
\Fun^{\lax.\mon}(([1],\min),(\Sigma^{(\infty,1)}_+(\ZZ_{\leq 0}),\otimes))
\arrow[leftrightarrow]{d}[sloped, anchor=south]{\sim}
\\
\Alg(\Ar(\Sigma^{(\infty,1)}_+(\ZZ_{\leq 0})), \boxdot)_{|\uno}
\arrow[hook]{r}{\ff}
\arrow[leftrightarrow]{d}[sloped, anchor=north]{\sim}
&
\Alg(\Ar(\Sigma^{(\infty,1)}_+(\ZZ_{\leq 0})), \boxdot)
\arrow[leftrightarrow]{d}[sloped, anchor=south]{\sim}
\\
\Alg(\Ar(\Sigma^{(\infty,1)}_+(\ZZ_{\leq 0})), \otimes)^{|\uno}
\arrow[hook]{r}{\ff}
\arrow{d}[sloped, anchor=north]{\sim}
&
\Alg(\Ar(\Sigma^{(\infty,1)}_+(\ZZ_{\leq 0})), \otimes)
\\
\Alg(\Sigma^{(\infty,1)}_+(\ZZ_{\leq 0}),\otimes)~,
\end{tikzcd}
\]
which we explain as follows.
\begin{itemize}

\item For any stably monoidal category $(\cD,\boxtimes)$, \Cref{lem.equivce.of.monoidal.strs.on.Ar.C} evidently upgrades to a commutative diagram
\[
\begin{tikzcd}
(\Ar(\cD),\boxdot)
\arrow[yshift=0.9ex]{rr}{\cofib}
\arrow[leftarrow, yshift=-0.9ex]{rr}[yshift=-0.2ex]{\sim}[swap]{\fib}
\arrow{rd}[sloped, swap]{t}
&
&
(\Ar(\cD),\boxtimes)
\arrow{ld}[sloped, swap]{s}
\\
&
(\cD,\boxtimes)
\end{tikzcd}
\]
among stably monoidal categories. Applying the functor $\Alg(\St) \xra{\Alg(-)} \Cat$ and taking fibers over the initial object $\uno \in \Alg(\cD,\boxtimes)$, we obtain an equivalence
\[
\begin{tikzcd}[column sep=2cm]
\Alg(\Ar(\cD),\boxdot)_{|\uno}
\arrow[yshift=0.9ex]{r}{\cofib}
\arrow[leftarrow, yshift=-0.9ex]{r}[yshift=-0.2ex]{\sim}[swap]{\fib}
&
\Alg(\Ar(\cD),\boxtimes)^{|\uno}
\simeq \Alg(\cD,\boxtimes)~,
\end{tikzcd}
\]
where $\Ar(\cD)_{|\uno}$ (resp. $\Ar(\cD)^{|\uno}$) denotes those arrows whose target (resp. source) is the initial object $\uno,$ and the equivalence $\Alg(\Ar(\cD),\boxtimes)^{|\uno}\simeq \Alg(\cD,\boxtimes)$ follows from the fact that $\uno$ is initial.
This explains the third, fourth, and fifth rows.

\item The uppermost vertical functors are postcomposition with the monoidal functor $(\ZZ_{\leq 0},+) \ra (\Sigma^{(\infty,1)}_+(\ZZ_{\leq 0}),\otimes)$.

\item The right vertical equivalence between the second and third rows follows from the universal property of Day convolution. Tracing through the definitions, one obtains its factorization as the left vertical equivalence: strictly unital functors correspond to algebra objects in $(\Ar(\Sigma^{(\infty,1)}_+(\ZZ_{\leq 0})),\boxdot)$ whose target satisfies the condition of being the initial object $\uno \in \Alg(\Sigma^{(\infty,1)}_+(\ZZ_{\leq 0}),\otimes)$.

\end{itemize}
Having constructed the monoidal functor $\varphi$, we obtain our desired exact monoidal extension
\[ \begin{tikzcd}
\bDelta_+
\arrow{r}{\varphi}
\arrow{d}
&
\Sigma^{(\infty,1)}_+(\ZZ_{\leq 0})
\\
\Sigma^{(\infty,1)}_+(\bDelta_+)
\arrow[dashed]{ru}[swap, sloped]{\tilde{\varphi}}
\end{tikzcd}
\]
by the universal property of the left adjoint $\Sigma^{(\infty,1)}_+$; it remains only to show that this functor is an equivalence in $\Alg(\St)$. Since the forgetful functor $\Alg(\St) \xra{\fgt} \St$ is conservative, it suffices to show that $\tilde{\varphi}$ is an equivalence on underlying stable categories. By \Cref{ldk.plus}, there is a distinguished equivalence
\[
\widetilde{\LDK_+}
\in
\Fun^\ex(\Sigma^{(\infty,1)}_+(\bDelta_+),\Sigma^{(\infty,1)}_+ \ZZ_{\leq 0})
\]
(the universal coaugmented Lurie--Dold--Kan equivalence), and it suffices to show that there exists an equivalence
\[
\tilde{\varphi}
\simeq
\widetilde{\LDK_+}
~.
\]
Applying \Cref{ldk.plus} again, we find that it suffices to establish the assignment
\[ \begin{tikzcd}[row sep=0cm]
\Fun^\ex(\Sigma^{(\infty,1)}_+(\bDelta_+),\Sigma^{(\infty,1)}_+ \ZZ_{\leq 0})
\arrow{r}{\sim}
&
\Fun(\bDelta_+,\Sigma^{(\infty,1)}_+ \ZZ_{\leq 0})
\arrow{r}{\LDK_+}[swap]{\sim}
&
\Fun(\ZZ_{\leq 0},\Sigma^{(\infty,1)}_+ \ZZ_{\leq 0})
\\
\rotatebox{90}{$\in$}
&
\rotatebox{90}{$\in$}
&
\rotatebox{90}{$\in$}
\\
\tilde{\varphi}
\arrow[maps to]{r}
&
\varphi
\arrow[dashed, maps to, bend left=15]{rd}
\\
\widetilde{\LDK_+}
\arrow[maps to]{rr}
&
&
h_\bullet~.
\end{tikzcd}
\]
Now, we observe that $\varphi$ is the coaugmented cobar complex of an algebra object structure on $\cofib(h_{-1} \ra h_0)$. (This is essentially a tautology: recall that an algebra object in a monoidal category is a monoidal functor from $\bDelta_+,$ and the underlying cosimplicial object --- i.e., the restriction of this functor to $\bDelta\subset \bDelta_+$ --- is the cobar complex.) Using the formula for $\LDK_+$ given in \Cref{ldk.plus}, the claimed assignment $\varphi \mapsto h_\bullet$ follows exactly as in the proof of \cite[Proposition 2.14]{MNN-nilp}.\footnote{Note that $\ZZ_{\leq 0}$ is the directed colimit of its subcategories $\ZZ_{(-n)//0}$ for $n \geq 0$, each of which is constructed from the former by adjoining a morphism. Hence, in order to construct an equivalence in $\Fun(\ZZ_{\leq 0},\Sigma^{(\infty,1)}_+(\ZZ_{\leq 0}))$, it suffices to construct it on objects and on minimal morphisms: there are no higher coherences to specify.}
\end{proof}

\begin{lemma}
\label{ldk.plus}
For any stable category $\cD$, there is a natural equivalence
\[
\Fun(\bDelta_+,\cD)
\xra[\sim]{\LDK_+}
\Fun(\ZZ_{\leq 0},\cD)
\]
which carries a coaugmented cosimplicial object $(X \ra Y^\bullet) \in \Fun(\bDelta_+,\cD)$ to the tower
\[
\cdots
\longra
\fib(X \longra \Tot_1(Y^\bullet))
\longra
\fib(X \longra \Tot_0(Y^\bullet))
\longra
X
~.
\]
\end{lemma}

\begin{proof}
We write
\[ \begin{tikzcd}[row sep=0cm]
\Fun(\bDelta,\cD)
\arrow{r}{\LDK}[swap]{\sim}
&
\Fun(\ZZ_{\leq 0},\cD)
\\
\rotatebox{90}{$\in$}
&
\rotatebox{90}{$\in$}
\\
Y^\bullet
\arrow[maps to]{r}
&
((-n) \longmapsto \Tot_n(Y^\bullet))
\end{tikzcd}
\]
for the Lurie--Dold--Kan equivalence \cite[Theorem 1.2.4.1]{Lurie-HA}. Observe that both categories $\bDelta$ and $\ZZ_{\leq 0}$ are weakly contractible, so that it is merely a condition for functors out of them to be constant; observe too that the equivalence $\LDK$ restricts to an equivalence between full subcategories of constant functors. Hence, we obtain the lower composite equivalence
\[ \begin{tikzcd}
\Fun([1],\Fun(\bDelta,\cD))
\arrow{r}{\Fun([1],\LDK)}[swap]{\sim}
&[1cm]
\Fun([1],\Fun(\ZZ_{\leq 0},\cD))
\arrow{r}{\fib}[swap]{\sim}
&
\Fun([1],\Fun(\ZZ_{\leq 0},\cD))
\\
\Fun(\bDelta^\lcone,\cD)
\arrow[hook]{u}
\arrow[dashed]{r}{\sim}
&
\Fun((\ZZ_{\leq 0})^\lcone,\cD)
\arrow[hook]{u}
\arrow[dashed]{r}{\sim}
&
\Fun((\ZZ_{\leq 0})^\rcone,\cD)
\arrow[hook]{u}
\end{tikzcd}
\]
among full subcategories on those natural transformations whose source, source, and target functors are respectively constant.\footnote{Here, $(-)^\lcone$ (resp.\! $(-)^\rcone$) denotes the left (resp.\! right) cone construction, given by adjoining an initial (resp.\! terminal) object. The weak contractibility e.g.\! of $\bDelta$ yields a localization $[1] \times \bDelta \ra \bDelta^\lcone$ at the morphisms whose first component is $\id_0$.} Applying the evident equivalences $\bDelta_+ \simeq \bDelta^\lcone$ and $(\ZZ_{\leq 0})^\rcone \simeq \ZZ_{\leq 0}$, this is precisely our desired equivalence $\LDK_+$.
\end{proof}

\begin{remark}
Let us write $\Env([1],\min)$ for the monoidal envelope of $([1],\min)$, and $\Env^\unital([1],\min)$ for its unitalization; the latter is the initial monoidal category equipped with a laxly monoidal strictly unital functor from $([1],\min)$. The proof of \Cref{cor.ideals.are.algebras} implicitly constructs a monoidal functor $\Env^\unital([1],\min) \ra (\ZZ_{\leq 0},+)$, and shows that it becomes an equivalence upon applying $\Sigma^{(\infty,1)}_+$. We believe that it is an equivalence unstably, but checking this appears to involve a nontrivial amount of combinatorics, so our proof proceeds by other means.
\end{remark}

\section{Spherical adjunctions}
\label{sec:sph}

We now use the work done in \Cref{section.algs.and.ideals} to prove our main results, culminating with our spectral description of the 2-category of spherical adjunctions (\Cref{sph.fctrs.are.modules.over.matrix.alg}).

\subsection{Definition and properties}

\begin{definition}
Given an adjunction
\begin{equation}
\label{generic.sphadj}
\begin{tikzcd}[column sep=2cm]
\cC_\Phi
\arrow[yshift=0.9ex]{r}{S}
\arrow[leftarrow, yshift=-0.9ex]{r}[yshift=-0.2ex]{\bot}[swap]{S^R}
&
\cC_\Psi
\end{tikzcd}
\end{equation}
in $\St$, consider the endomorphisms
\[
\hspace{-0.5cm}
T_\Phi
:=
T_{S,\Phi}
:=
\fib \left(
\id_{\cC_\Phi}
\xlongra{\eta}
S^R S
\right)
\in
\End_\St(\cC_\Phi)
\quad
\text{and}
\quad
T_{\Psi}
:=
T_{S,\Psi}
:=
\cofib \left(
S S^R
\xlongra{\varepsilon}
\id_{\cC_\Psi}
\right)
\in
\End_\St(\cC_\Psi)~.
\]
We say that the adjunction $S \adj S^R$ is a \bit{spherical adjunction} if the endomorphisms $T_{\Phi}$ and $T_{\Psi}$ are automorphisms, in which case we refer to them as (\bit{spherical}) \bit{twists}.\footnote{It is reasonable to write $T_{\Phi,R} := T_\Phi$ and $T_{\Psi,R} := T_\Psi$, and thereafter to denote the inverses by $T_{\Phi,L} := (T_{\Phi,R})^{-1}$ and $T_{\Psi,L} := (T_{\Psi,L})^{-1}$ since they are similarly extracted from the adjunction $S^L \adj S$ of \Cref{obs.basics.on.sphadjns}\Cref{item.have.ladjt}.} We write
\[
\Sph
\subset
\Fun(\Adj,\St)
\]
for the full sub-2-category on the spherical adjunctions. 
By \cite{RV-adj}*{Theorem 4.4.18}, an adjunction is equivalent to the data of its left adjoint functor,
so we may equivalently denote a spherical adjunction simply by its left adjoint, which we call its underlying \bit{spherical functor} (as e.g.\! in the notation $T_{S,\Phi}$ and $T_{S,\Psi}$ above). Moreover, we write $\Phi,\Psi \in \Fun(\Sph,\St)$ for the functors carrying a spherical adjunction to its source and target (i.e., to the source and target of its underlying spherical functor). So for instance, given a spherical adjunction \Cref{generic.sphadj}, we may write $\cC_\Phi = \Phi(S)$.
\end{definition}

\begin{observation}
\label{obs.basics.on.sphadjns}
We will require the following basic facts about spherical adjunctions; see \cite{DKSS} for more details.
\begin{enumerate}

\item

The functors $S$ and $S^R$ intertwine the twists, up to a shift:
\[
S T_\Phi
\simeq
\Sigma^{-2} T_\Psi S
\quad
\text{and}
\quad
S^R T_\Psi
\simeq
\Sigma^2 T_\Phi S^R
~.
\]

\item\label{item.have.ladjt}

Given a spherical adjunction $S \adj S^R$, the functor $S$ admits a left adjoint $S^L$, given by the formulas
\[
S^L
\simeq
\Sigma^{-1} T_\Phi^{-1} S^R
\simeq
\Sigma S^R T_\Psi^{-1}
~.
\]
Moreover, the adjunction $S^L \adj S$ is also spherical. Reversely, we can write $S^R$ in terms of $S^L$ by the formulas
\[
S^R
\simeq
\Sigma T_\Phi S^L
\simeq
\Sigma^{-1} S^L T_\Psi
~.
\]
It follows that $S^R$ is also a spherical functor.

\end{enumerate}
\end{observation}

\begin{definition}
We define the \bit{Fourier transform} to be the autoequivalence
\[
\begin{tikzcd}[row sep=0cm, column sep=2cm]
\Sph
\arrow{r}{\fF}[swap]{\sim}
&
\Sph
\\
\rotatebox{90}{$\in$}
&
\rotatebox{90}{$\in$}
\\
(S \adj S^R)
\arrow[maps to]{r}
&
(S^L \adj S)
\end{tikzcd}
\]
taking a spherical functor $S$ to its left adjoint $S^L$: this definition appeals to both the result \Cref{thm:adjunctions.are.adjoints} that a 2-category of adjunctions is equivalent to a 2-category of left adjoint functors, and the calculation of \Cref{obs.basics.on.sphadjns}\Cref{item.have.ladjt} that the left adjoint of a spherical functor is itself spherical. 

In other words, the Fourier transform of a spherical adjunction \Cref{generic.sphadj} is the spherical adjunction
\[
\begin{tikzcd}[column sep=2cm]
\cC_\Psi
\arrow[yshift=0.9ex]{r}{S^L}
\arrow[leftarrow, yshift=-0.9ex]{r}[yshift=-0.2ex]{\bot}[swap]{S}
&
\cC_\Phi~.
\end{tikzcd}
\]
(In particular, we have equivalences $\Phi \circ \fF \simeq \Psi$ and $\Psi \circ \fF \simeq \Phi$ in $\Fun(\Sph,\St)$.) When we refer to a spherical adjunction by its underlying spherical functor $\cC_\Phi\xra{S}\cC_\Psi,$ we may treat $\fF(S)$ and $S^L$ as synonyms for the same spherical adjunction.
\end{definition}

\begin{remark}
Given a spherical adjunction $(\cC_\Phi\xra{S}\cC_\Psi) \in \Sph$, the twists of its Fourier transform are
\[
T_{\fF(S),\Phi}
\simeq
T_{S,\Psi}^{-1}
\quad
\text{and}
\quad
T_{\fF(S),\Psi}
\simeq
T_{S,\Phi}^{-1}
~.
\]
We refer to \cite{DKSS}*{Corollary 2.5.16} for a detailed calculation.
\end{remark}

\subsection{The universal spherical adjunction}\label{subsec:univ}
Recall
from \Cref{subsection:notn.and.conventions}
the categories $f\cV$ and $g\cV$ of filtered and graded vector spaces.
%
%
%
The latter category is equivalent to $\Coh(B\GG_m)\simeq{\sf Rep}(\GG_m)$ (since a representation of $\GG_m$ is just a vector space graded by its weight-spaces for the $\GG_m$ action) and the former also has a geometric interpretation via the Rees construction, as we now describe.
Consider the free commutative $\kk$-algebra as an object $\kk[x] \in \CAlg(g\cV)$ by declaring that $x$ has weight $-1$. 
Then there is a symmetric monoidal equivalence (described for instance in \cite{Lurie-Rotation}*{Proposition 3.1.6})
\[
\begin{tikzcd}[row sep=0cm]
f\cV \arrow[r, "\sim"] & \Perf_{\kk[x]}(g\cV) \\
\rotatebox{90}{$\in$}
&
\rotatebox{90}{$\in$}
\\
h_n = h_0 \langle n \rangle \arrow[r, mapsto] & \kk[x]\langle n \rangle
\end{tikzcd}
\]
presenting $f\cV$ as the category $\Perf_{\kk[x]}(g\cV)$ of graded $\kk[x]$-modules, or equivalently as the category $\Coh(\AA^1/\GG_m)$ of $\GG_m$-equivariant coherent sheaves on $\AA^1.$ (See also \cite{BZ-Gunningham}*{\S 5.2} for a more detailed discussion of this latter identification.)

\begin{observation}
\label{obs.gr.is.spherical}
The adjunction
\[
\begin{tikzcd}[column sep=2cm]
f\cV = \Perf_{\kk[x]}(g \cV)
\arrow[yshift=0.9ex]{r}{\gr}
\arrow[leftarrow, yshift=-0.9ex]{r}[yshift=-0.2ex]{\bot}[swap]{\triv}
&
\Perf_{\kk}(g \cV) = g\cV
\end{tikzcd}
\]
corresponding to the augmentation $\kk[x] \to \kk$ is spherical: its twist automorphisms are
\[
T_\Phi = \langle -1 \rangle = (-) \otimes h_{-1} \in \Aut_\St(f\cV)
\quad
\text{and}
\quad
T_\Psi = \Sigma^2 \langle -1 \rangle = \Sigma^2 (-) \otimes \delta_{-1} \in \Aut_\St(g\cV)
~.
\]
\end{observation}

\begin{theorem}
\label{thm.gr.coreps.LHS}
The object $\gr \in \Sph$ corepresents the functor $\Sph \xra{\Phi} \St$: evaluation at the object $h_0 \in f\cV$ defines an equivalence
\[
\hom_\Sph((f\cV\xra{\gr}g\cV),(\cC_\Phi\xra{S}\cC_\Psi))
\xlongra{\sim}
\cC_\Phi
~.
\]
The inverse sends an object $c \in \cC_{\Phi}$ to the morphism $F^c: \gr \to S$ with components
\[
\begin{tikzcd}[row sep=0cm]
f\cV
\arrow{r}{F^c_\Phi}
&
\cC_{\Phi}
\\
\rotatebox{90}{$\in$}
&
\rotatebox{90}{$\in$}
\\
h_n
\arrow[mapsto]{r}
&
(T_{\Phi})^{-n} c
\end{tikzcd}
\quad
\text{and}
\quad
\begin{tikzcd}[row sep=0cm]
g\cV
\arrow{r}{F^c_\Psi}
&
\cC_{\Psi}
\\
\rotatebox{90}{$\in$}
&
\rotatebox{90}{$\in$}
\\
\delta_n
\arrow[mapsto]{r}
&
\Sigma^{2n} (T_{\Psi})^{-n}Sc~.
\end{tikzcd}
\]
\end{theorem}

\begin{proof}
We recall from \Cref{lem:KL-EM} that pullback along the inclusion $\Mnd\hookra \Adj$ has a left adjoint
\begin{equation}
\label{adjn.Kleisli.is.ladjt}
\begin{tikzcd}[column sep=2cm]
\Fun(\Mnd,\St)
\arrow[yshift=0.9ex, hook]{r}{\Kl}
\arrow[leftarrow, yshift=-0.9ex]{r}[yshift=-0.2ex]{\bot}[swap]{\fgt}
&
\Fun(\Adj,\St)~,
\end{tikzcd}
\end{equation}
the {\em Kleisli construction}.
Below, we will be interested in the adjunction $\gr:f\cV\adjarr g\cV:\triv,$ and we observe that the left adjoint $f\cV \xra{\gr} g\cV$ is essentially surjective, so that by \Cref{prop:kl-of-surj} the adjunction $\gr\adj \triv$ is in the image of $\Kl$. In other words, we may recover $\gr\adj \triv$ as the Kleisli adjunction of its inderlying monad:
\begin{equation}
\label{equation.gr.and.triv.is.Kleisli}
\Kl(f\cV , \triv \circ \gr)
\simeq
(\gr \adj \triv)
~.
\end{equation}

Note also that $\Mnd := \fB \bDelta_+$, where as usual (see \Cref{notation.appendix}) we write $\fB$ of a monoidal category to denote the corresponding 1-object 2-category; combining this presentation of $\Mnd$ with \Cref{cor.ideals.are.algebras} we obtain an equivalence
\[
\Sigma^{(\infty,2)}_+ \Mnd
\simeq
\Sigma^{(\infty,2)}_+ \fB \ZZ_{\leq 0}
\]
in $\Stu_{\kk,2}$, which yields the composite equivalence
\begin{equation}
\label{composite.equivce.betw.monads.and.nonpos.integer.actions.in.St}
\Fun(\Mnd,\St)
\simeq
\Fun^{2-\ex}(\Sigma^{(\infty,2)}_+ \Mnd , \St)
\simeq
\Fun^{2-\ex}(\Sigma^{(\infty,2)}_+ \fB \ZZ_{\leq 0} , \St)
\simeq
\Fun( \fB \ZZ_{\leq 0} , \St )
~.
\end{equation}
By \Cref{prop:BN-localization}, the inclusion $\ZZ_{\leq 0} \hookra \ZZ$ is the monoidal localization at the object $(-1) \in \ZZ_{\leq 0}$, so that by restriction we obtain a full subcategory
\[
\Fun ( \fB \ZZ , \St)
\subseteq
\Fun ( \fB \ZZ_{\leq 0} , \St)
~.
\]
Moreover, by definition (namely the assumption that the twist $T_\Phi$ is invertible) we have a factorization
\begin{equation}
\label{factorization.of.Sph.to.Z.actions}
\begin{tikzcd}
\Fun(\fB \ZZ_{\leq 0} , \St)
\arrow[leftrightarrow]{r}{\sim}
&
\Fun( \Mnd , \St)
\arrow[leftarrow]{r}{\fgt}
&
\Fun(\Adj,\St)
\\
\Fun(\fB \ZZ , \St)
\arrow[hook]{u}
&
&
\Sph~,
\arrow[dashed]{ll}
\arrow[hook]{u}
\end{tikzcd}
\end{equation}
since the upper-left corner of \Cref{factorization.of.Sph.to.Z.actions} sends the generating morphism in $\fB\ZZ_{\leq 0}$ to the twist functor; andinvertibility of the twist in $\Sph$ combined with the presentation of $\ZZ_{\geq 0}\to \ZZ$ as a monoidal localization at this object guarantees the factorization as above.

We now conclude by fixing an object $(\cC_\Phi\xra{S}\cC_\Psi) \in \Sph$ and observing the natural composite equivalence
\begin{align}
\label{composite.equivalence.from.hom.Sph.gr.C.bullet.step.1}
\hom_\Sph ( \gr , S )
& \simeq
\hom_{\Fun(\Adj,\St)} ( \gr \adj \triv , S \adj S^R)
\\
\label{composite.equivalence.from.hom.Sph.gr.C.bullet.step.2}
& \simeq
\hom_{\Fun(\Adj,\St)} ( \Kl(f\cV , \triv \circ \gr) , S \adj S^R)
\\
\label{composite.equivalence.from.hom.Sph.gr.C.bullet.step.3}
& \simeq
\hom_{\Fun(\Mnd,\St)} ( (f\cV , \triv \circ \gr) , (\cC_\Phi,S^R S))
\\
\label{composite.equivalence.from.hom.Sph.gr.C.bullet.step.4}
& \simeq
\hom_{\Fun(\fB \ZZ_{\leq 0} , \St)} ( ( f\cV , \langle -1 \rangle) , (\cC_\Phi , T_\Phi ) )
\\
\label{composite.equivalence.from.hom.Sph.gr.C.bullet.step.5}
& \simeq
\hom_{\Fun(\fB \ZZ , \St)} ( ( f\cV , \langle -1 \rangle) , (\cC_\Phi , T_\Phi ) )
\\
\label{composite.equivalence.from.hom.Sph.gr.C.bullet.step.6}
& \simeq
\hom_\St ( \cV , \cC_\Phi )
\\
\label{composite.equivalence.from.hom.Sph.gr.C.bullet.step.7}
& \simeq
\cC_\Phi
~,
\end{align}
in which
\begin{itemize}

\item equivalence \Cref{composite.equivalence.from.hom.Sph.gr.C.bullet.step.1} follows from the definition of $\Sph \subset \Fun(\Adj,\St)$ as a full subcategory,

\item equivalence \Cref{composite.equivalence.from.hom.Sph.gr.C.bullet.step.2} follows from equivalence \Cref{equation.gr.and.triv.is.Kleisli},

\item equivalence \Cref{composite.equivalence.from.hom.Sph.gr.C.bullet.step.3} follows from adjunction \Cref{adjn.Kleisli.is.ladjt},

\item equivalence \Cref{composite.equivalence.from.hom.Sph.gr.C.bullet.step.4} follows from equivalence \Cref{composite.equivce.betw.monads.and.nonpos.integer.actions.in.St},

\item equivalence \Cref{composite.equivalence.from.hom.Sph.gr.C.bullet.step.5} follows from the factorization \Cref{factorization.of.Sph.to.Z.actions},

\item equivalence \Cref{composite.equivalence.from.hom.Sph.gr.C.bullet.step.6} follows from the adjunction $\St \adjarr \Fun(\fB \ZZ , \St)$ (which is nothing more than the base-change adjunction $\Mod_{\kk}(\St)\adjarr \Mod_{f\cV}(\St)$) and the fact that its left adjoint acts as $\cV \mapsto (f\cV,\langle -1 \rangle)$, and

\item equivalence \Cref{composite.equivalence.from.hom.Sph.gr.C.bullet.step.7} follows from the fact that $\cV \in \St$ is free on an object.

\end{itemize}
Tracing through these equivalences, we see that the composite is indeed implemented by evaluation at $h^0 \in f\cV$, as desired.
\end{proof}

\begin{remark}\label{rem:microlocalization}
In this remark, which will not be used elsewhere in the paper, we motivate a connection to symplectic geometry by explaining how the constructions involved in the proof of \Cref{thm.gr.coreps.LHS} may be understood as categorifications of some well-known functors among perverse sheaves.

The adjoint triple
\[ \begin{tikzcd}[column sep=1.5cm]
\Fun(\Mnd,\St)
\arrow[bend left]{r}{\Kl}
\arrow[leftarrow]{r}[description, pos=0.45]{\ev_{\ell}}[yshift=1.5ex, pos=0.45]{\bot}[swap, yshift=-1.5ex, pos=0.45]{\bot}
\arrow[bend right]{r}[swap]{{\EM}}
&
\Fun(\Adj,\St)
\end{tikzcd}
\]
constructed in \Cref{lem:KL-EM}
induces an adjoint triple
\begin{equation}
\label{adjt.triple.ModfVSt.Sph}
\begin{tikzcd}[column sep=1.5cm]
\Mod_{f\cV}(\St)
\arrow[bend left]{r}{\Kl}
\arrow[leftarrow]{r}[description, pos=0.35]{\Phi}[yshift=1ex, pos=0.35]{\bot}[swap, yshift=-1ex, pos=0.35]{\bot}
\arrow[bend right]{r}[swap]{{\EM}}
&
\Sph~,
\end{tikzcd}
\end{equation}
using the fact (established during the proof of \Cref{thm.gr.coreps.LHS}) that the functor $\Phi:\Sph\to \St$ factors through $\Mod_{f\cV}(\St)$. Furthermore, the ``free module'' functor $F: g\cV \to f\cV \cong \Perf_{\kk[x]}(g\cV)$ is symmetric monoidal and hence induces a change-of-base adjoint triple
\begin{equation}
\label{adjt.triple.ModgvSt.ModfVSt}
\begin{tikzcd}[column sep=1.5cm]
\Mod_{g\cV}(\St)
\arrow[bend left]{r}{f\cV\otimes_{g\cV}(-)}
\arrow[leftarrow]{r}[description, pos=0.5]{F^*}[yshift=1.2ex, pos=0.5]{\bot}[swap, yshift=-1.2ex, pos=0.5]{\bot}
\arrow[bend right]{r}[swap]{\Hom_{\Mod_{g\cV}(\St)}(f\cV,-)}
&
\Mod_{f\cV}(\St)~.
\end{tikzcd}
\end{equation}
The composition of the two adjoint triples \Cref{adjt.triple.ModfVSt.Sph} and \Cref{adjt.triple.ModgvSt.ModfVSt} categorifies the well-known adjoint triple
\begin{equation}\label{eq:perverse-triple} \begin{tikzcd}[column sep=1.5cm]
\Loc(\CC^\times)
\simeq
&[-1.8cm]
\Perv(\CC^\times)
\arrow[bend left]{r}{j_!}
\arrow[leftarrow]{r}[description, pos=0.63]{j^*}[yshift=1.2ex, pos=0.63]{\bot}[swap, yshift=-1.2ex, pos=0.63]{\bot}
\arrow[bend right]{r}[swap]{j_*}
&
\Perv(\CC, \cS_{toric})
\end{tikzcd}
\end{equation}
coming from the open affine embedding $j: \CC^\times \hookrightarrow \CC$.
\end{remark}

\begin{remark}\label{rem:segal}
    The idea of using the Kleisli construction to produce 
    a spherical functor, categorifying the top arrow in \Cref{eq:perverse-triple}, has appeared earlier in \cite{Segal-autoequivalences}. However, the construction sketched in \cite{Segal-autoequivalences}*{\S 3.2} takes as input a category equipped with only a coaugmented autoequivalence, rather than a full invertible Smith ideal. The former structure is not sufficient to give a unique construction of a spherical adjunction, leading to the counterexample described in \cite{Christ-spherical}*{Example 4.10} of two monads with the same underlying coaugmented twist autoequivalence but different multiplications. The main idea behind the proof of \Cref{thm.gr.coreps.LHS} in the present paper is that the full Smith ideal structure is necessary to produce spherical adjunctions via the Kleisli construction.
\end{remark}

\begin{observation}\label{obs:sigma}
Let $\sigma\in \Aut(g\cV)$ denote the symmetric monoidal autoequivalence of $g\cV$ defined as the composite
\[
\begin{tikzcd}[row sep=0cm]
g\cV \arrow[r, "\mathsf{rev}"] & g\cV \arrow[r, "\mathsf{shear}"] & g\cV \\
\delta_n \arrow[r, mapsto] & \delta_{-n} \arrow[r, mapsto] & \Sigma^{-2n} \delta_{-n}
\end{tikzcd}
\]
of grading-reversal and the ``shearing'' autoequivalence \cite{Ga-shcat}*{\S 13.4} which shifts the homological degree according to the grading. The automorphism $\sigma$
is symmetric monoidal and  intertwines $\langle n \rangle$ with $\Sigma^{-2n} \langle -n \rangle$. Let $\kk[\beta] = \sigma(\kk[x])$. Because $\sigma$ sends the augmentation of $\kk[x]$ to an augmentation of $\kk[\beta],$ it induces an equivalence of spherical adjunctions $\gr \simeq \sigma(\gr),$ depicted vertically in the following diagram:
\[
\begin{tikzcd}[column sep=2cm]
f\cV = \Perf_{\kk[x]}(g \cV)
\arrow[yshift=0.9ex]{r}{\gr}
\arrow[leftarrow, yshift=-0.9ex]{r}[yshift=-0.2ex]{\bot}[swap]{\triv}
\arrow[d, "\sigma"]
&
\Perf_{\kk}(g \cV) = g\cV \arrow[d, "\sigma"]
\\
\sigma(f\cV) = \Perf_{\kk[\beta]}(g \cV)
\arrow[yshift=0.9ex]{r}{\sigma(\gr)}
\arrow[leftarrow, yshift=-0.9ex]{r}[yshift=-0.2ex]{\bot}[swap]{\sigma(\triv)}
&
\Perf_{\kk}(g \cV) = g\cV~.
\end{tikzcd}
\]
The twist automorphisms of the spherical functor $\sigma(\gr)$ are
\[
T_\Phi = \Sigma^{-2} \langle 1 \rangle \in \Aut_\St(\sigma(f\cV))
\quad
\text{and}
\quad
T_\Psi = \langle 1 \rangle \in \Aut_\St(g\cV)
~.
\]
See \cite{Lurie-Rotation}*{\S 3.3} for a more detailed study of the category $\sigma(f\cV),$ which is denoted by $\Theta$ there.
\end{observation}


\begin{corollary}
\label{cor.corep.RHS}
The object $\fF(\sigma(\gr)) \in \Sph$ corepresents the functor $\Sph \xra{\Psi} \St$: evaluation at the object $\kk[\beta] \in \sigma(f\cV)$ defines an equivalence
\[
\hom_\Sph(\fF(\sigma(f\cV)\xra{\sigma(\gr)}g\cV),\cC_\Phi \xra{S}\cC_\Psi)
\xlongra{\sim}
\cC_\Psi
~.
\]
The inverse sends an object $c \in \cC_{\Psi}$ to the morphism $G^c: \gr \to S$ with components
\[
\begin{tikzcd}[row sep=0cm]
\sigma(f\cV)
\arrow{r}{G^c_\Psi}
&
\cC_{\Psi}
\\
\rotatebox{90}{$\in$}
&
\rotatebox{90}{$\in$}
\\
\kk[\beta]\langle n \rangle
\arrow[mapsto]{r}
&
\Sigma^{2n} (T_{\Psi})^{-n} c
\end{tikzcd}
\quad
\text{and}
\quad
\begin{tikzcd}[row sep=0cm]
g\cV
\arrow{r}{G^c_\Phi}
&
\cC_{\Phi}
\\
\rotatebox{90}{$\in$}
&
\rotatebox{90}{$\in$}
\\
\delta_n
\arrow[mapsto]{r}
&
(T_{\Phi})^{-n}S^Rc~.
\end{tikzcd}
\]
\end{corollary}

\begin{proof}
This follows from \Cref{thm.gr.coreps.LHS} via the equivalences
\[
\hom_\Sph(\fF(\sigma(\gr)),S)
\simeq
\hom_\Sph(\sigma(\gr),\fF^{-1}(S))
\simeq
\hom_\Sph(\gr,\fF^{-1}(S))
\simeq
\Phi(\fF^{-1}(S))
\simeq
\cC_\Psi
~.
\qedhere
\]
\end{proof}

\begin{observation} \label{obs:gvtwist}
We also could have used $\fF(\gr)$ or $\fF^{-1}(\gr)$ in \Cref{cor.corep.RHS} to corepresent the functor $\Psi.$
Our choice of $\fF(\sigma(\gr))$ is meant
to ensure that the twist automorphism $T_{\Phi}$ of the left-hand category $g\cV$ is $\langle -1 \rangle,$
which will be essential in the proof of \Cref{prop:spectral-description}.
\end{observation}



\begin{remark}\label{universal.spherical.adjunction}
Tautologically, there exists a natural adjunction
\begin{equation}\label{eq:tautological-adjunction}
\begin{tikzcd}[column sep=2cm]
 \Phi
\arrow[yshift=0.9ex]{r}{\mathbf{S}}
\arrow[leftarrow, yshift=-0.9ex]{r}[yshift=-0.2ex]{\bot}[swap]{\mathbf{S}^R}
&
\Psi
\end{tikzcd}
\end{equation}
internal to the 2-category of functors $\Fun(\Sph,\St)$. Using \Cref{thm.gr.coreps.LHS} and \Cref{cor.corep.RHS}, we see that the adjunction \eqref{eq:tautological-adjunction} is corepresented by an adjunction
\begin{equation}\label{eq:universal-adjunction}
\begin{tikzcd}[column sep=2cm]
\gr
\arrow[leftarrow, yshift=0.9ex]{r}{\mathbf{S}_{univ}}
\arrow[yshift=-0.9ex]{r}[yshift=-0.2ex]{\top}[swap]{(\mathbf{S}^R)_{univ}}
&
\fF(\sigma(\gr))
\end{tikzcd}
\end{equation}
internal to $\Sph$ which may be called the \textit{universal spherical adjunction} as it corepresents the identity functor $\Sph \xra{\id} \Sph$.
(Note that because the universal spherical adjunction is a corepresenting object, the handedness of the adjunctions \Cref{eq:universal-adjunction} and \Cref{eq:tautological-adjunction} is reversed: $(\mathbf{S}^R)_{univ}$ is {\em left} adjoint to $\mathbf{S}_{univ}.$)
In the notation of \Cref{thm.gr.coreps.LHS} and \Cref{cor.corep.RHS}, the maps underlying this adjunction are $(\mathbf{S}^R)_{univ} = F^{\delta_0}$ and $\mathbf{S}_{univ} = G^{\delta_0}$.
\end{remark}

\begin{observation}\label{universal.twist}
There is a natural automorphism $\mathbf{T}: \id_{\Sph} \xra{\sim} \id_{\Sph}$ which acts on a spherical adjunction $\cC_\Phi\xra{S}\cC_\Psi$ as
\begin{equation}\label{eq:aside-monodromy} \begin{tikzcd}[column sep=2cm]
\cC_\Phi
\arrow[yshift=0.9ex]{r}{S}
\arrow[leftarrow, yshift=-0.9ex]{r}[yshift=-0.2ex]{\bot}[swap]{S^R}
\arrow[dashed]{d}[swap]{\mathbf{T}_{\Phi} = T_\Phi}
&
\cC_\Psi
\arrow[dashed]{d}{\mathbf{T}_{\Psi} = \Sigma^{-2} T_\Psi}
\\
\cC_\Phi
\arrow[yshift=0.9ex]{r}{S}
\arrow[leftarrow, yshift=-0.9ex]{r}[yshift=-0.2ex]{\bot}[swap]{S^R}
&
\cC_\Psi~.
\end{tikzcd}
\end{equation}
(That this is a map of spherical adjunctions follows from \Cref{obs.basics.on.sphadjns} (1).) We can think of this map as the \textit{universal twist automorphism}. 
In the spirit of \Cref{universal.spherical.adjunction}, we may also understand $\mathbf{T}$ through its corepresenting map, namely the automorphism $\mathbf{T}_{univ}$ of the spherical functor $\gr\oplus \fF(\sigma(\gr))$
(commuting with the universal spherical adjunction described in \Cref{universal.spherical.adjunction})
which we may write componentwise (in the notation of \Cref{thm.gr.coreps.LHS} and \Cref{cor.corep.RHS}) as
\[
\mathbf{T}_{univ} = \left( \begin{array}{cc}
F^{h_{-1}}
&
0
\\
0
&
 G^{\kk[\beta]\langle -1 \rangle}
\end{array}
\right).
\]
\end{observation}

\subsection{A spectral description of spherical adjunctions}
\label{subsec:spectral}

\begin{notation}\label{notation:algebra-a}
We denote by $\cA := \End_\Sph(\gr \oplus \fF(\sigma(\gr))) \in \Alg(\St)$ the endomorphism algebra of the direct sum.
\end{notation}

\begin{theorem}
\label{sph.fctrs.are.modules.over.matrix.alg}
The functor
\begin{equation}
\label{the.fctr.from.Sph.to.Mod.A.St}
\Sph
\xra{\hom_\Sph(\gr \oplus \fF(\sigma(\gr)),-)}
\Mod_\cA(\St)
\end{equation}
is an equivalence.
\end{theorem}

\begin{proof}
We first prove that the functor \Cref{the.fctr.from.Sph.to.Mod.A.St} is an equivalence on $\iota_1$, which amounts to proving that the functor
\begin{equation}
\label{monadic.forgetful.functor.on.iota.one}
\iota_1 \Sph
\xra{\hom_\Sph ( \gr \oplus \fF(\sigma(\gr)) , - )}
\iota_1 \St
\end{equation}
is the right adjoint of a monadic adjunction: in other words, that its source $\iota_1 \Sph$ is the category of modules in its target $\iota_1 \St$ for the monad $\cA \otimes (-)$, with \Cref{monadic.forgetful.functor.on.iota.one} being the forgetful functor.
Note that by \Cref{thm.gr.coreps.LHS} and \Cref{cor.corep.RHS} we can identify the functor \Cref{monadic.forgetful.functor.on.iota.one} as the composite
\begin{equation}
\label{monadic.forgetful.functor.on.iota.one.as.a.composite}
\iota_1 \Sph
\xra{(\ev_\ell,\ev_r)}
\iota_1 \St \times \iota_1 \St
\xlongra{\oplus}
\iota_1 \St
~.
\end{equation}

In order to establish monadicity, we verify the hypotheses of the Barr--Beck--Lurie theorem \cite[Theorem 4.7.3.5]{Lurie-HA}: it suffices to show that the functor \Cref{monadic.forgetful.functor.on.iota.one} is conservative, admits a left adjoint, and preserves colimits.
Conservativity is clear from the description \Cref{monadic.forgetful.functor.on.iota.one.as.a.composite}. For the remaining properties, we will observe that $\iota_1\Sph$ is a functor category: namely, if we write $(\Sigma^{(\infty,1)}_+ \Adj)'$ for the localization of the universal stably-enriched adjunction which freely adjoints inverses to the twists, then $\Sph = \Fun((\Sigma^{(\infty,1)}_+\Adj)', \St).$

As a consequence, we know from \Cref{lem:functorcats.presentability} that $\Sph$ is a presentable stable 2-category, with tensors given by pointwise product, and colimits in $\iota_1\Sph$ computed pointwise. By the definition of tensors, tensoring against the object $\gr\oplus \fF(\sigma(\gr))$ provides a left adjoint to \Cref{monadic.forgetful.functor.on.iota.one}; moreover, since colimits are computed pointwise, 
it is clear that the functor 
\Cref{monadic.forgetful.functor.on.iota.one} preserves colimits.

We have thus shown that the functor \Cref{the.fctr.from.Sph.to.Mod.A.St} between 2-categories is an equivalence on $\iota_1$. Moreover, its source and target both admit cotensors over small categories, computed pointwise. Hence, in order to lift this equivalence on $\iota_1$ to an equivalence of 2-categories, it suffices to observe that the functor \Cref{the.fctr.from.Sph.to.Mod.A.St} commutes with cotensors.
\end{proof}

\begin{remark}\label{rem:matrix-naive}
The monoidal category $\cA$ is the endomorphism category of the direct sum $\gr\oplus\sigma(\gr)$ and can therefore be represented as the categorical $2\times 2$ matrix algebra 
\begin{equation*}
\cA\simeq
\left(\begin{array}{cc}
\hom_\Sph(\gr,\gr)&\hom_\Sph(\gr,\fF(\sigma(\gr)))\\
\hom_\Sph(\fF(\sigma(\gr)),\gr)&\hom_\Sph(\fF(\sigma(\gr)),\fF(\sigma(\gr)))
\end{array}
\right)
\end{equation*}
in $\Alg(\St),$ where the right-hand side has underlying category given by the direct sum of its entries, with monoidal structure given by matrix multiplication (where the off-diagonal entries have the obvious bimodule structure for the monoidal categories on the diagonals).
Using \Cref{thm.gr.coreps.LHS} and \Cref{cor.corep.RHS}, we may rewrite this presentation of $\cA$ as 
\[
\cA
\simeq
\left( \begin{array}{cc}
f\cV
&
g\cV
\\
g\cV
&
\sigma(f\cV)
\end{array}
\right),
\]
where the bimodule structures on the off-diagonal entries are given by the
functors $\gr \adj \triv$ and $\sigma(\gr) \adj \sigma(\triv)$.
\end{remark}

In the rest of this section we will give a spectral description of the monoidal category $\cA$ which will clarify the meaning of the automorphism $\sigma$. Our description will use the monoidal structure on the category of coherent sheaves $\Coh(X\times_YX)$ on a fiber product of the form $X\times_YX$: 
namely, the convolution diagram
\[
\begin{tikzcd}
&X \times_Y X \times_Y X \arrow[dl, "p_{12}"']\arrow[d, "p_{13}"]\arrow[dr, "p_{23}"]&\\
X\times_Y X& X\times_Y X & X\times_Y X
\end{tikzcd}
\]
can be used to define a monoidal product $M\star N:= (p_{13})_*(p_{12}^*M\otimes p_{13}^*N).$ (See for instance \cite{BZFN} for more discussion.) When $X=\bigsqcup_i X_i$ is a disjoint union of components $X_i,$ we will write the monoidal category $\Coh(X\times_YX)$ as a categorical matrix algebra (as in \Cref{rem:matrix-naive}) with $(i,j)$-entry given by $\Coh(X_i\times_Y X_j).$ Equivalently, for the sake of cleaner notation, we will sometimes write this monoidal category as $\Coh(\mathcal{M}),$ where $\mathcal{M}$ is the matrix of varieties or stacks whose $(i,j)$ entry is $X_i\times_Y X_j.$

\begin{proposition}\label{prop:spectral-description}
There is an equivalence of monoidal categories
\begin{align*}
\cA&\simeq
\Coh\left(
(\AA^1/\GG_m \sqcup 0/\GG_m)
\times_{\AA^1/\GG_m}
(\AA^1/\GG_m \sqcup 0/\GG_m)
\right)\\
&\simeq \Coh\left(
\begin{array}{ll}
\AA^1/\GG_m \times_{\AA^1/\GG_m} \AA^1/\GG_m & \AA^1/\GG_m \times_{\AA^1/\GG_m} 0/\GG_m\\
0/\GG_m \times_{\AA^1/\GG_m} \AA^1/\GG_m & 0/\GG_m \times_{\AA^1/\GG_m} 0/\GG_m
\end{array}
\right)\\
&\simeq \Coh\left(
\begin{array}{ll}
\AA^1/\GG_m & 0/\GG_m\\
0/\GG_m & \AA^1[-1]/\GG_m
\end{array}
\right)~,
\end{align*}
where the right-hand side is given the convolution monoidal structure defined above.
\end{proposition}

In order to prove \Cref{prop:spectral-description}, we first provide an alternative description of the spectral convolution monoidal category in the proposition.
Observe that the categories $f\cV$ and $g\cV$ may be upgraded to stable module categories for the monoidal category $f\cV$: the former as the rank-1 free module, and the latter through the functor $f\cV\xra{\gr}g\cV.$

\begin{lemma}\label{lem:BZNP}
There is an equivalence of monoidal categories
\begin{align}\label{eq:spectral-matrix-description}
\Coh\left(
(\AA^1/\GG_m \sqcup 0/\GG_m)
\times_{\AA^1/\GG_m}
(\AA^1/\GG_m \sqcup 0/\GG_m)
\right)
 &\simeq 
\End_{\Mod_{f\cV}(\St)}(f\cV\oplus g\cV)~.
\end{align}
\end{lemma}
\begin{proof}
We apply \cite{BZNP}*{Proposition 1.2.1}, which gives, for $X\to Y$ a proper map of perfect stacks with $X$ smooth, an equivalence
\begin{equation}\label{eq:integral-kernels-proposition}
\Coh(X\times_Y X)\simeq \End_{\Mod_{\Perf(Y)}(\St)}(\Coh(X))
\end{equation}
by treating the left-hand side as a category of integral kernels. If we assume that $Y$ is also smooth, we may replace $\Perf(Y)$ by the equivalent category $\Coh(Y).$

Now we specialize to the case 
\[
X := (\AA^1/\GG_m\sqcup 0/\GG_m) \to \AA^1/\GG_m =: Y~.
\]
The standard identifications $\Coh(B\GG_m) \simeq g\cV $ and $\Coh(\AA^1/\GG_m) \simeq f\cV$ (discussed at the beginning of \Cref{subsec:univ}) give an identification $\Coh(X)\simeq f\cV\oplus g\cV$. Combining this with \Cref{eq:integral-kernels-proposition}, we have an equivalence
\[
\Coh(X\times_Y X)\simeq \End_{\Mod_{f\cV}}(f\cV\oplus g\cV)~.\qedhere
\]
\end{proof}

\begin{proof}[Proof of \Cref{prop:spectral-description}]
In the proof of \Cref{thm.gr.coreps.LHS} we showed that the functor $\Phi\in \Fun(\Sph,\St)$ can be lifted to an object of $\Fun(\Sph,\Mod_{f\cV}(\St))$. In other words, the left-hand category in a spherical adjunction carries an $f\cV$-module structure that is respected by maps of spherical adjunctions.

Applying this enhanced evaluation functor to the endomorphisms of the spherical adjunction $\gr\oplus \fF(\sigma(\gr))$, we obtain a monoidal functor
\begin{align}
\cA = \End_{\Sph}(\gr\oplus\fF(\sigma(\gr)))
&\to \End_{\Mod_{f\cV}(\St)}(\Phi(\gr)\oplus\Phi(\fF(\gr)))
\nonumber
\\
&\simeq
\End_{\Mod_{f\cV}(\St)}(f\cV\oplus g\cV)~.
\label{eq:A-to-fvmods}
\end{align}
Using the description of the twists in \Cref{obs.gr.is.spherical} and \Cref{obs:gvtwist} we see that the $f\cV$-module structures on $f\cV$ and $g\cV$ above are given by $\id_{f\cV}$ and $\gr,$ respectively. We can therefore compose \Cref{eq:A-to-fvmods} with the equivalence \Cref{eq:spectral-matrix-description} to obtain a monoidal functor
\[
\cA
\to
\Coh\left(
(\AA^1/\GG_m \sqcup 0/\GG_m)
\times_{\AA^1/\GG_m}
(\AA^1/\GG_m \sqcup 0/\GG_m)
\right)~.
\]
As in \Cref{rem:matrix-naive}, we can think of this as a monoidal functor
\begin{equation}\label{eq:A-to-spectral}
\left( \begin{array}{cc}
f\cV
&
g\cV
\\
g\cV
&
\sigma(f\cV)
\end{array}
\right)
\to
\Coh\left(\begin{array}{ll}
\AA^1/\GG_m &  0/\GG_m \\
0/\GG_m & \AA^1[-1]/\GG_m
\end{array}\right)
\end{equation}
of $2\times 2$ matrix algebras in $\St$. To show that \Cref{eq:A-to-spectral} is an equivalence it is sufficient to observe that it induces an equivalence on each of the four components. Indeed, three of these maps are the canonical identifications discussed previously, and the
lower-right corner is the standard Koszul duality equivalence
\begin{equation}\label{eq:koszul}
\sigma(f\cV) = \Perf_{\kk[\beta]}(g\cV)\xra{\sim} \Coh_{\kk[\epsilon]}(g\cV) = \Coh(\AA^1[-1]/\GG_m)~,
\end{equation}
relating perfect (resp. coherent) modules over a symmetric (resp. exterior) algebra with the appropriate shifts, where we write $\kk[\epsilon]$ for the ring of functions on $\AA^1[-1].$ 
\end{proof}

\begin{observation}\label{spectral.univ.twist.calc}
Under the equivalence of \Cref{prop:spectral-description}, the universal twist automorphism $\mathbf{T}_{univ} \in \End_\Sph(\gr\oplus \fF(\sigma(\gr)))= \cA$ computed in \Cref{universal.twist} is sent to the object
\[
\left( \begin{array}{cc}

\cO_{\AA^1/\GG_m}\langle -1 \rangle &
0
\\
0
&
\delta_{0/\GG_m}\langle -1 \rangle
\end{array}
\right)
\in
\Coh\left(\begin{array}{ll}
\AA^1/\GG_m &  0/\GG_m \\
0/\GG_m & \AA^1[-1]/\GG_m
\end{array}\right).
\]
\end{observation}

\section{Torus actions and toric stacks}\label{sec:higherdim}
In order to simplify our notation, in this section we specialize to coefficient field $\kk = \CC.$\footnote{Of course, the appropriate statements remain true for general coefficients. But we believe the following discussion is made clearer and simpler by treating the A-side, which is always ``geometric'' (i.e., defined over $\CC$) on equal footing with the B-side, which requires taking $\CC$ coefficients.}

The 2-categories we have studied so far (and some generalizations which we are to discuss soon) admit actions by tori, and we will study invariants for these actions. For instance, understanding the $S^1$-action on the 2-category $\Sph$ and the procedure of taking invariants will allow us to give a notion of ``$S^1$-equivariant spherical functors,'' which are associated to the symplectic geometry of ${\cotangent}(\CC/\CC^\times)$ rather than that of ${\cotangent}\CC.$ Indeed, as we shall see in \Cref{cor:ungauging-dim1}, the mirror spectral description of this 2-category $\Sph^{S^1}$ will be stated in terms of ${\cotangent} \CC$ rather than ${\cotangent}(\CC/\CC^\times).$

We begin in \Cref{subsec:ambidexterity}, by establishing a technical result identifying limits and colimits of groupoid-shaped diagrams of presentable enriched categories. We apply this in \Cref{subsec:gauging} to deduce an equivalence between invariants and coinvariants for $G$-actions on 2-categories. Finally, in \Cref{subsec:hypertoric} we apply the preceding results to describe the 2-category of $S^1$-invariant spherical functors and higher-dimensional generalizations.

\subsection{Presentable ambidexterity}\label{subsec:ambidexterity}

\begin{notation}
Given a 2-category $\cX$, we write $\cX^L \subseteq \cX$ for the 1-full sub-2-category on the left adjoint 1-morphisms.\footnote{This is consistent with the notation $\PrL$ if one takes $\Pr \subset \what{\Cat}$ to denote the full sub-2-category on the presentable categories (or alternatively the 1-full sub-2-category on the accessible functors between presentable categories).} We write
\[
\cX
\xlonghookla{\fgt}
\cX^L
\xhookra{\radjt}
\cX^{1\&2-\op}
\]
for the evident forgetful functors (which are both inclusions of 1-full sub-2-categories).
\end{notation}

\begin{theorem}
\label{thm.prbl.ambidext}
Fix any presentably symmetric monoidal category $\cW \in \CAlg(\Pr^L)$, any 1-category $\cA$, and any diagram $\cA \xra{F} \Pr^L_\cW$ of presentable $\cW$-enriched categories. If the category $\cA$ is a groupoid, then there is a canonical equivalence
\[
\colim_\cA^{\Pr^L_\cW}(F)
\simeq
\lim_\cA^{\Pr^L_\cW}(F)
~.
\]
More precisely, there is a canonical dashed functor making the diagram
\[ \begin{tikzcd}
&[2cm]
&
\Pr^L_\cW
\\
\Spaces_{/\Pr^L_\cW}
\arrow[dashed]{r}[description]{{\sf co}/\lim^{\Pr^L_\cW}}
\arrow[bend left=10]{rru}[sloped]{\colim^{\Pr^L_\cW}}
\arrow[bend right=10]{rrd}[sloped, swap]{\lim^{\Pr^L_\cW}}
&
\Pr^{LL}_\cW
\arrow{ru}[sloped, swap]{\fgt}
\arrow{rd}[sloped]{\radjt}
\\
&
&
(\Pr^L_\cW)^{1\&2-\op}
\end{tikzcd}
\]
commutative.
\end{theorem}

\begin{remark}
A similar ambidexterity result appears in \cite{Stefanich-Pres}*{Theorem 1.3.3}.
\end{remark}

\begin{proof}
We first address the unenriched case: i.e., we set $\cW := \Spaces = \uno_{\Pr^L}$. For this, we recall the following constructions and functorialities of limits and colimits in $\Pr^L$ (indexed over 1-categories), which follow from the evident equivalence $\Pr^L \simeq (\Pr^R)^{1\&2-\op}$ along with the fact that both forgetful functors $\Pr^L \xra{\fgt} \what{\Cat}$ and $\Pr^R \xra{\fgt} \what{\Cat}$ commute with limits \cite[Proposition 5.5.3.13 and Theorem 5.5.3.18]{Lurie-HTT}.
\begin{itemize}

\item

The limit of a functor $\cA \xra{F} \Pr^L$ can be computed (in $\what{\Cat}$) as the category of cocartesian sections of the presentable fibration corresponding to $F$ (i.e., the cocartesian unstraightening of the composite $\cA \xra{F} \Pr^L \xra{\fgt} \what{\Cat}$). Moreover, given a morphism
\begin{equation}
\label{morphism.in.Cat.over.PrL}
\begin{tikzcd}
\cA
\arrow{rd}[sloped, swap]{F}
\arrow{rr}{\varphi}
&
&
\cB
\arrow{ld}[sloped,swap]{G}
\\
&
\Pr^L
\end{tikzcd}
\end{equation}
in $\Cat_{/\Pr^L}$, the canonical morphism
\[
\lim^{\Pr^L}_\cA(F)
\longla
\lim^{\Pr^L}_\cB(G)
\]
in $\Pr^L$ can be computed (in $\what{\Cat}$) as the restriction functor between categories of cocartesian sections.

\item

The colimit of a functor $\cA \xra{F} \Pr^L$ can be computed as the limit of the corresponding functor $\cA^\op \simeq \cA^{1\&2-\op} \xra{F^{1\&2-\op}} (\Pr^L)^{1\&2-\op} \simeq \Pr^R$, which can be computed (in $\what{\Cat}$) as the category of cartesian sections of the presentable fibration corresponding to $F$ (which is also the cartesian unstraightening of the composite $\cA^\op \xra{F^{1\&2-\op}} \Pr^R \xra{\fgt} \what{\Cat}$). Moreover, given a morphism \Cref{morphism.in.Cat.over.PrL} in $\Cat_{/\Pr^L}$, the canonical morphism
\[
\colim^{\Pr^L}_\cA(F)
\longra
\colim^{\Pr^L}_\cB(G)
\]
in $\Pr^L$ is the left adjoint of the restriction functor between categories of cartesian sections.

\end{itemize}
The claim (in the case that $\cW := \Spaces$) now follows immediately from the observation that given a presentable fibration over a groupoid, all sections are both cocartesian and cartesian.

We now turn to the general case. For this, we use the equivalence $\Pr^L_\cW \simeq \Mod_\cW(\Pr^L)$. Observe that the forgetful functor $\Mod_\cW(\Pr^L) \xra{\fgt} \Pr^L$ admits both a left and a right adjoint (namely $\cW \otimes (-)$ and $\Fun^L(\cW,-)$, respectively). In particular, because it has a left adjoint, limits in $\Mod_\cW(\Pr^L)$ can be computed in $\Pr^L$ -- i.e., as cocartesian sections of the corresponding presentable fibration, with residual $\cW$-action computed fiberwise. In order to deduce the corresponding result for colimits in $\Mod_\cW(\Pr^L)$, introduce the commutative diagram
\begin{equation}\label{enriching.diagram} \begin{tikzcd}
\Mod_\cW(\Pr^L)
\arrow[leftrightarrow]{r}{\sim}
\arrow{d}[swap]{\fgt}
&
\coMod_\cW(\Pr^R)^{1\&2-\op}
\arrow{d}{\fgt}
\\
\Pr^L
\arrow[leftrightarrow]{r}[swap]{\sim}
&
(\Pr^R)^{1\&2-\op}
\end{tikzcd}
\end{equation}
(considering $\cW$ as an object of $\CAlg(\Pr^L) \simeq \CAlg((\Pr^R)^{1\&2-\op}) \simeq \coCAlg(\Pr^R)^{1\&2-\op}$). The fact that the left vertical functor in \Cref{enriching.diagram} admits both adjoints implies that the right vertical functor does as well. Hence, limits in $\coMod_\cW(\Pr^R)$ can be computed in $\Pr^R$ -- i.e., as cartesian sections of the corresponding presentable fibration, with residual $\cW$-coaction computed fiberwise. Combining these observations with the unenriched case, the claim now follows.
\end{proof}

\begin{corollary}
\label{cor.ambidext.adjn}
The functor
\[
\Spaces^\op
\xra{\Fun(-,\Pr^L_\cW)}
\what{\what{\Cat}}
\]
takes values in ambidextrous adjunctions: for any morphism $X \xra{f} Y$ in $\Spaces$, there is a canonical equivalence $f_! \simeq f_*$ between the left and right adjoints to the functor $f^* = \Fun(f,\Pr^L_\cW)$.
\end{corollary}

\begin{proof}
The functor $\Fun(X,\Pr^L_\cW) \xla{f^*} \Fun(Y,\Pr^L_\cW)$ admits both adjoints, given by left and right Kan extension (note that $\Pr^L_\cW$ admits all limits and colimits). Because $X \xra{f} Y$ is both a cocartesian fibration and a cartesian fibration, these Kan extensions are respectively computed by fiberwise colimit and fiberwise limit. As the fibers are groupoids, the claim follows from \Cref{thm.prbl.ambidext}.
\end{proof}

\subsection{Gauging and ungauging}\label{subsec:gauging}

We begin by recalling the notions of torus-equivariant A- and B-type 2-categories. Throughout, we fix a homomorphism $G \xra{f} H$ of tori, and write $G^\vee \xla{f^\vee} H^\vee$ for its dual.

\begin{definition}
\label{defn.Loc.three}
A \bit{local system of 2-categories over $BG$} is a presentable stable 2-category equipped with a topological $G$-action. These assemble into the 3-category
\[
\Loc^{(3)}(BG)
:=
\Fun (BG_{\sf{B}} , \Pr^{L,\st_{\kk}}_2 )
~,
\]
where we write $BG_{\sf{B}}$ for the classfying space of the complex torus $G$, thought of as an $\infty$-groupoid. (The subscript $(-)_{\sf{B}}$ is meant to be suggestive of the Betti stack construction.)
There is a natural functor
\[
(B f) ^* : \Loc^{(3)}(BG) \to \Loc^{(3)}(BH)
\]
given by pullback along the map 
$BG\xra{Bf} BH$
of groupoids.
We write $(B f)_! \adj (B f)^* \adj (B f)_*$ for its left and right adjoints. (See \Cref{lem:loc3.adjoints} for the construction of these adjoints.)
\end{definition}

\begin{remark}\label{rem:loc3-from-e2}
Since $G\simeq B\pi_1 G$ is a $K(\pi_1 G, 1),$ an object of $\Loc^{(3)}(BG)$ is equivalently specified as the data of a presentable stable 2-category $\cC \in \Pr^{L,\st_{\kk}}_2$ equipped with an $\EE_2$ homomorphism $\pi_1 G \to \hom(\id_\cC,\id_\cC).$ 
\end{remark}

\begin{definition}\label{defn:3qcoh}
A \bit{quasicoherent sheaf of 2-categories over $BG^\vee$} is a presentable stable 2-category equipped with an action of $\fB(\Perf(BG^\vee))$. These assemble into the 3-category
\[
\QCoh^{(3)}(BG^\vee)
:=
\Fun^{3-\ex} (\fB^2(\Perf(BG^\vee), \otimes), \Pr^{L,\st_{\kk}}_2 )
~,
\]
where we write $\Fun^{3-ex}$ for the internal Hom in $\Stu_{\kk,3},$ as described in \Cref{subsec:stable-ncats}.
The 3-category $\QCoh^{(3)}(BG^\vee)$ is most naturally covariant in the variable $G^\vee$: a homomorphism of tori $H^\vee \xra{f^\vee} G^\vee$ determines a morphism $\Perf(B H^\vee) \xla{(Bf^\vee)^*} \Perf(B G^\vee)$ in $\CAlg(\St)$, and thereafter we obtain a functor \[
(Bf^\vee)_*: \QCoh^{(3)}(B H^\vee) \ra \QCoh^{(3)}(B G^\vee)
\]
by precomposition. As in the A-side case, this functor admits both adjoints, which we denote by $(B f^\vee)^*\adj (B f^\vee)_* \adj (B f^\vee)^!.$
\end{definition}

\begin{remark}
There is a more general definition of the 3-category $\QCoh^{(3)}(X)$ for a stack $X$ (which can be found in \cite{Stefanich-QCoh}). The 3-category $\QCoh^{(3)}(BG^\vee)$ of \Cref{defn:3qcoh} agrees with that more general definition (up to size considerations) for 2-affine stacks. (See \cite{Stefanich-QCoh}*{Theorem 14.3.9} for 2-affineness of $BG^\vee.$)
\end{remark}

\begin{example}\label{ex:linear-coh-actions}
Let $\cA$ be a stably monoidal category. The 2-category $\Mod_{\cA}(\St)$ can be made an object of $\QCoh^{(3)}(BG)$ by 
specifying an $\mathbb{E}_2$ homomorphism $\Perf(BG)\to Z(\cA)$ into the categorical center of $\cA,$ thus
promoting $\cA$ to an algebra object in $\QCoh^{(2)}(BG)$.
One source of such central homomorphisms may be found in the following lemma.
\end{example}

\begin{lemma}\label{lem:e2-constloops}
Let $p:X\to Y$ be a proper surjective map of smooth perfect stacks, and $\cA$ the convolution category
\[
\cA:= \Coh(X\times_Y X) \simeq \End_{\Mod_{\Coh(Y)}(\St)}(\Coh(X))~,
\]
where the equivalence follows from \cite{BZNP}*{Proposition 1.2.1} as in the proof of \Cref{lem:BZNP}. Then the functor
\begin{equation}\label{eq:diag-e2-functor}
\Delta_*p^*:\Coh(Y)\to \Coh(X\times_YX)~,
\end{equation}
where $\Delta: X\to X\times_Y X$ is the diagonal map, admits a canonical lift to an $\EE_2$ homomorphism into the center $Z(\cA)$ of the algebra $\cA.$
\end{lemma}
\begin{proof}
The center $Z(\cA)$ is computed in \cite{BZNP}*{Theorem 1.2.10}:
there is an equivalence of $\EE_2$-categories
\[
Z(\cA)\simeq \Coh_{prop/Y}(LY)~,
\] 
where $LY:=Y\times_{Y\times Y}Y$ is the algebraic loop space of $Y$ and $\Coh_{prop/Y}(LY)$ denotes the category of coherent sheaves on $LY$ whose pushforward to $Y$ is coherent.

Pushforward along the inclusion of constant loops $i:Y\to LY$ defines an $\EE_2$ map 
\[
i_*:\Coh(Y)\to \Coh_{prop/Y}(LY)~,
\]
and the resulting functor $\Coh(Y)\to Z(\cA)$ has an underlying central functor $\Coh(Y)\to \cA$ given by \eqref{eq:diag-e2-functor}.
\end{proof}

%

The main result of this subsection is an identification of the respective 3-categories associated to topological actions of $G$ and algebraic actions of $G^\vee$. This multiply-categorified Fourier transform (a categorification of the Fourier transform described in \Cref{ex:intro-cat-fourier}) can be understood as ``fully extended Betti Langlands'' (in the sense of an equivalence of 3-categories) for the group $G$:
\begin{theorem}\label{thm:4d-fourier}
There is an equivalence of 3-categories
\begin{equation}\label{eq:4d-fourier}
\Loc^{(3)}(BG) \simeq \QCoh^{(3)}(BG^\vee)
\end{equation}
which is coherently functorial for maps of tori: 
in particular, given a homomorphism $f: G \to H$ of tori (with dual map $f^\vee:G^\vee \leftarrow H^\vee$), the diagram
\begin{equation}\label{eq:3d-fourier-functoriality}
\begin{tikzcd}
\Loc^{(3)}(BH)
\arrow[leftrightarrow]{r}{\sim}
\arrow{d}[swap]{(B f)^*}
&
\QCoh^{(3)}(BH^\vee)
\arrow{d}{(B f^\vee)_*}
\\
\Loc^{(3)}(BG)
\arrow[leftrightarrow]{r}[swap]{\sim}
&
\QCoh^{(3)}(BG^\vee)
\end{tikzcd}
\end{equation}
commutes.
\end{theorem}
\begin{proof}
Observe the equivalence
\begin{equation}\label{eq:perfbgv}
\Sigma^{(\infty,1)}_+ ( \pi_1 G )
\simeq
\Perf(BG^\vee)
\end{equation}
in $\CAlg(\St)$; indeed, the right side, the category $\mathsf{Rep}^{\mathsf{fin}}(BG^\vee)$ of finite-dimensional representations of $BG^\vee$, is semisimple, with simple objects indexed by the characters of $G^\vee$, or equivalently the cocharacters of $G$, or equivalently $\pi_1 G$. This yields the last equivalence in the composite equivalence
\[
\Sigma^{(\infty,3)}_+ BG
\simeq
\Sigma^{(\infty,3)}_+ B^2 (\pi_1 G)
\simeq
\fB^2 \Sigma^{(\infty,1)}_+ (\pi_1 G)
\simeq
\fB^2 \Perf(BG^\vee)
~,
\]
and thereafter applying $\Fun^{3-\ex}(-,\Pr^{L,\st_{\kk}}_2)$ yields the equivalence \Cref{eq:4d-fourier}. From here, the square \Cref{eq:3d-fourier-functoriality} canonically commutes because the square
\[
\begin{tikzcd}
\Sigma^{(\infty,1)}_+ ( \pi_1 H )
\arrow[leftrightarrow]{r}{\sim}
&
\Perf(B H^\vee)
\\
\Sigma^{(\infty,1)}_+ ( \pi_1 G )
\arrow[leftrightarrow]{r}[swap]{\sim}
\arrow{u}
&
\Perf(B G^\vee)
\arrow{u}
\end{tikzcd}
\]
canonically commutes.
\end{proof}


The equivalence \eqref{eq:4d-fourier} therefore also intertwines the right adjoints of the vertical maps in \eqref{eq:3d-fourier-functoriality}. However, we will find that we would like to establish that the right adjoint of the
left vertical map $(B f)^*$
is intertwined with the {\em left} adjoint of the 
right vertical map $(B f^\vee)_*$
In other words, we will show that the left and right adjoints of these maps agree. This is a consequence of the presentable ambidexterity theorem proven above.

%
%
%
%
%

\begin{lemma}\label{lem:bside-ambidexterity}
The right and left adjoints of the functor 
\[
(B f^\vee)_*: \QCoh^{(3)}(BH^\vee)\to \QCoh^{(3)}(BG^\vee)
\]
are equivalent.
\end{lemma}
\begin{proof}
This follows from \Cref{thm:4d-fourier} and \Cref{cor.ambidext.adjn}. (Note that $\Pr^{L,\st_{\kk}}_2 := \Pr^L_\St$.)
\end{proof}

%
%

The utility of \Cref{lem:bside-ambidexterity} is that the left adjoint $(f^\vee)^*$ of $(f^\vee)_*$ is a categorified deequivariantization.
\begin{observation}
The following diagram commutes:
\[
\begin{tikzcd}[column sep=2cm]
\Alg(\QCoh^{(2)}(BH^\vee)) \arrow[r, "(B f^\vee)^{*}"] \arrow[d, "\Mod_{(-)}(\St)"'] & \Alg(\QCoh^{(2)  }(BH^\vee)) \arrow[d, "\Mod_{(-)}(\St)"]  \\
\QCoh^{(3)}(BH^\vee) \arrow[r, "(B f^\vee)^*"'] & \QCoh^{(3)}(BG^\vee)~.
\end{tikzcd}
\]
\end{observation}

\begin{lemma}\label{lem:algebraic-ungauging}
Consider an iterated pullback square
\[
\begin{tikzcd}
\bar{X} \arrow[r, "\bar{p}"] \arrow[dr, phantom, "\ulcorner", very near start] \arrow[d, "b"] & \bar{Y} \arrow[r, "\bar{q}"] \arrow[dr, phantom, "\ulcorner", very near start] \arrow[d, "a"] & B G^\vee \arrow[d, "Bf^\vee"] \\
X \arrow[r, "p"] & Y \arrow[r, "q"] & B H^\vee ~,
\end{tikzcd}
\]
where $p$ is a proper map between smooth algebraic stacks. There is an isomorphism of stably monoidal categories
\[
\Coh(\bar{X} \times_{\bar{Y}} \bar{X}) \simeq (B f^\vee)^*\left( \Coh(X \times_{Y} X) \right).
\]

\begin{proof}
Given a map 
of algebraic stacks, the 2-categories $\QCoh^{(2)}$ admit $*$-pushforward and *-pullback functors satisfying base change.
(See \cite{diFiore}*{Chapter 5} or \cite{Stefanich-QCoh}*{Chapter 14} for more details.) Using this, we obtain the sequence of equivalences
\begin{align}
\label{Coh.is.pullback.equivce.one}
(B f^\vee)^* \left(\Coh(X \times_{Y} X ) \right)
&\simeq
(B f^\vee)^* \left( q_* \ul{\hom}_{\QCoh^{(2)}(Y)}(p_* \Perf(X) , p_* \Perf(X)) \right)
\\
\label{Coh.is.pullback.equivce.two}
&\simeq
\bar{q}_* a^* \ul{\hom}_{\QCoh^{(2)}(Y)}(p_* \Perf(X) , p_* \Perf(X))
\\
\nonumber
&\simeq
\bar{q}_* \ul{\hom}_{\QCoh^{(2)}(\bar{Y})}(a^* p_* \Perf(X) , a^* p_* \Perf(X))
\\
\label{Coh.is.pullback.equivce.three}
&\simeq
\bar{q}_* \ul{\hom}_{\QCoh^{(2)}(\bar{Y})}(\bar{p}_* b^* \Perf(X) , \bar{p}_* b^* \Perf(X))
\\
\nonumber
&\simeq
\bar{q}_* \ul{\hom}_{\QCoh^{(2)}(\bar{Y})}(\bar{p}_* \Perf(\bar{X}) , \bar{p}_* \Perf(\bar{X}))
\\
\label{Coh.is.pullback.equivce.four}
&\simeq
\Coh(\bar{X} \times_{\bar{Y}} \bar{X})
~,
\end{align}
in which equivalences \Cref{Coh.is.pullback.equivce.one} and \Cref{Coh.is.pullback.equivce.four} follow from \cite{BZNP}*{Proposition 1.2.1} while equivalences \Cref{Coh.is.pullback.equivce.two} and \Cref{Coh.is.pullback.equivce.three} follow from base change.
\end{proof}
\end{lemma}

\subsection{Toric stacks}\label{subsec:hypertoric}

We now apply the preceding results to a monodromy action on $\Sph,$ and then to a generalization $\Sph_n$ of $\Sph$ which has $n$ commuting monodromy actions.

\begin{notation}
Up until \Cref{notation.fix.ntorus.D.and.subtorus.G.and.quotient.F}, we write $G := \CC^\times.$
\end{notation}

\begin{proposition}\label{prop:monodromy-intertwining}
The 2-category $\Sph$ can be promoted to an object of $\QCoh^{(3)}(BG^\vee)$ so that the structure of $\Sph$ as an object of $\Loc^{(3)}(BG)$ (under the equivalence of \Cref{thm:4d-fourier}) has underlying automorphism of $\id_{\Sph}$ given by the universal twist automorphism
\begin{equation}\label{eq:aside-monodromy2} \begin{tikzcd}[column sep=2cm]
\cC_\Phi
\arrow[yshift=0.9ex]{r}{S}
\arrow[leftarrow, yshift=-0.9ex]{r}[yshift=-0.2ex]{\bot}[swap]{S^R}
\arrow[dashed]{d}[swap]{T_\Phi}
&
\cC_\Psi
\arrow[dashed]{d}{\Sigma^{-2} T_\Psi}
\\
\cC_\Phi
\arrow[yshift=0.9ex]{r}{S}
\arrow[leftarrow, yshift=-0.9ex]{r}[yshift=-0.2ex]{\bot}[swap]{S^R}
&
\cC_\Psi
\end{tikzcd}
\end{equation}
described in \Cref{universal.twist}.
\end{proposition}
\begin{proof}
Let $X = (\CC\sqcup 0)/G^\vee$, $Y=\CC/G^\vee.$ From \Cref{sph.fctrs.are.modules.over.matrix.alg}, \Cref{prop:spectral-description}, and \Cref{lem:BZNP}, we may
present the 2-category $\Sph$ as a module 2-category $\Sph\simeq \Mod_{\cA}(\St),$ where
\begin{equation}\label{eq:spectralalg}
\cA := \Coh(X\times_{Y} X)~.
\end{equation}
\Cref{lem:e2-constloops} provides a homomorphism
\begin{equation}\label{eq:central-hom-to-a}
\fZ:\Perf(BG^\vee)\to Z(\cA)~,
\end{equation}
thus promoting $\Mod_\cA(\St)$ to an object of $\QCoh^{(3)}(BG^\vee)\simeq \Loc^{(3)}(BG).$ As described in \Cref{rem:loc3-from-e2}, this is equivalent to the data of an $\EE_2$ map $\pi_1G\to \End(\id_{\Sph}),$ and we would like to check that this map sends a generator of $\pi_1 G =\ZZ$ to \eqref{eq:aside-monodromy2}.


Due to our choice of sign in \Cref{def.smith.ideal}, the generator of interest to us is the one which is naturally written as $-1\in \ZZ.$ 
Under the equivalence \Cref{eq:perfbgv}, the generator $-1\in \pi_1G$ corresponds to the weight-$(-1)$ representation of $G^\vee,$ or equivalently to the coherent sheaf $\cO_{BG^\vee}\langle -1\rangle \in \Perf(BG^\vee).$ Therefore,
the action of this generator under \Cref{eq:central-hom-to-a} is the automorphism given by multiplication with the central object
\[
\gamma:= \fZ(\cO_{BG^\vee}\langle -1 \rangle)~,
\]
which we now compute.

Recall that $\fZ$ can be computed as
the functor
$\Delta_* p^*,$ where $p$ is the map $X\to Y$ and $\Delta$ is the diagonal $X\to X\times_{Y} X.$ Since $\Delta$ factors through the inclusion
$$
\left(\CC/G^\vee\times_{\CC/G^\vee}\CC/G^\vee\right)
\sqcup 
\left(0/G^\vee\times_{\CC/G^\vee}0/G^\vee\right)
\longra
X\times_Y X~,
$$
the object $\gamma$ is a diagonal object in the matrix category $\cA$: we may write $\gamma=\gamma_{\Phi}\oplus \gamma_{\Psi}$ for its upper-left and lower-right components, respectively.
 
 The object $\gamma_{\Phi}$ is the image of $\cO_{BG^\vee}
\langle -1 \rangle$ under the functor
 \[
 \Coh(BG^\vee)
 \longra
 \Coh(\CC/G^\vee\times_{\CC/G^\vee}\CC/G^\vee)\simeq\Coh(\CC/G^\vee)
 \]
of pullback along the map $p:\CC/G^\vee\to BG^\vee.$ 
Since the pullback preserves the twist $\langle -1\rangle$ and takes the structure sheaf to the structure sheaf, we have $\gamma_{\Phi}=p^*\cO_{BG^\vee}\langle -1\rangle = \cO_{\CC/G^\vee}\langle -1 \rangle$. (This is the object of $f\cV$ which we earlier denoted $h_{-1}$.)

Similarly, $\gamma_{\Psi}$ is given by the image of $\cO_{BG^\vee}\langle -1 \rangle$ under the functor
\begin{equation}\label{eq:central-psi}
\Coh(BG^\vee)
\longra
\Coh(0/G^\vee\times_{\CC/G^\vee} 0/G^\vee)\simeq\Coh(\CC[-1]/G^\vee)~,
\end{equation}
where the first functor in \Cref{eq:central-psi} is the pushforward under the diagonal, so that the composite functor is the pushforward under the inclusion $0/G^\vee\to \CC[-1]/G^\vee.$
This functor therefore sends $\cO_{BG^\vee}\langle -1\rangle$ to $\delta_{0/\CC^\times}\langle -1 \rangle,$ the ($\langle -1 \rangle$-twisted) skyscraper sheaf at the origin of $\CC[-1]/G^\vee$.

We conclude that the object $\gamma$ is presented by the matrix
\[
\gamma = \left( \begin{array}{cc}
\cO_{\CC/\CC^\times}\langle -1 \rangle &
0
\\
0
&
\delta_{0/\CC^\times}\langle -1 \rangle
\end{array}
\right)
\in
\Coh\left(\begin{array}{ll}
\CC/\CC^\times &  0/\CC^\times \\
0/\CC^\times & \CC[-1]/\CC^\times
\end{array}\right)~;
\]
as described in \Cref{spectral.univ.twist.calc}, this is the matrix describing the universal twist automorphism \Cref{eq:aside-monodromy2}.
\end{proof}

\begin{remark}
As we have checked in \Cref{universal.twist}, 
the map \eqref{eq:aside-monodromy2} does define an invertible object of $\End(\id_{\Sph})$ and therefore an $\EE_1$ map $\pi_1(G)\simeq\ZZ\to \End(\id_{\Sph}).$ However, to promote $\Sph$ to an object of $\Loc^{(3)}(BG)$ would require upgrading this $\EE_1$ map to an $\EE_2$ map, which is difficult to accomplish manually; this is why we proved \Cref{prop:monodromy-intertwining} via applying our main theorem and passing to the spectral description of $\Sph.$ However, a sufficiently functorial theory of perverse schobers would have given us this $\EE_2$ map for free, coming from the geometric $S^1$ action on $(\CC,0).$ We can thus interpret \Cref{prop:monodromy-intertwining} as evidence for a rich yet-to-be-developed structure on perverse schobers. (This $S^1$ action may also be manifest in the description of spherical functors as paracyclic Segal categories conjectured in \cite{DKSS}.)
\end{remark}

As an immediate corollary of \Cref{prop:monodromy-intertwining}, we see that passing to $G$-invariants on the A-side undoes the $G^\vee$-quotient on the B-side:
\begin{corollary}\label{cor:ungauging-dim1}
The 2-category $\Sph^G$ of $G$-invariant spherical functors is equivalent to the 2-category of modules over the monoidal convolution category
$
\Coh\left(
(\CC \sqcup 0)
\times_{\CC}
(\CC \sqcup 0)
\right).
$
\end{corollary}
\begin{proof}
Let $p:G\to\{1\}$ be the trivial group homomorphism, and $p^\vee:\{1\}\to G^\vee$ the dual homomorphism. For $\cC$ an object of $\Loc^{(3)}(BG),$ the $G$-invariant 2-category $\cC^G$ is the image of $\cC$ under the pushforward 
\[
(Bp)_*:\Loc^{(3)}(BG)\to \Loc^{(3)}(B\{1\})\simeq\PrLSt_2.
\]
By \Cref{lem:bside-ambidexterity}, the equivalence $\Loc^{(3)}(BG)\simeq \QCoh^{(3)}(BG^\vee)$ of \Cref{thm:4d-fourier} intertwines $(Bp)_*$ with the pullback
\[(Bp^\vee)^*:\QCoh^{(3)}(BG^\vee)\to \QCoh^{(3)}(B\{1\})\simeq \PrLSt_2.\]
The equivalence of \Cref{thm:4d-fourier} identifies $\Sph\in \Loc^{(3)}(BG)$ with the 2-category of modules for $\cA=\Coh^{G^\vee}\left((\CC\sqcup 0)\times_\CC (\CC\sqcup 0)\right),$ and we may conclude from \Cref{lem:algebraic-ungauging} that the image of this 2-category under the pullback $(Bp^\vee)^*$ is the 2-category of modules for the deequivariantized monoidal category $\Coh\left((\CC\sqcup 0)\times_\CC (\CC\sqcup 0)\right).$
\end{proof}

\begin{remark}
Intuitively, an object of $\Sph^G$ may be understood as a spherical adjunction equipped with a trivialization of the monodromy automorphism \Cref{eq:aside-monodromy2}. The canonical maps $T_{\Phi} \to \id_{\cC_\Phi}$ and $\id_{\cC_\Psi} \to T_{\Psi}$ then become the natural transformations
\[
\id_{\cC_\Phi} \xra{x} \id_{\cC_\Phi} \qquad \text{and} \qquad \Sigma^{-2} \id_{\cC_\Psi} \xra{\beta} \id_{\cC_\Psi}~.
\]
underlying the $\Coh(\CC) \simeq \Perf_{\kk[x]}(\cV)$-module structure and $\Coh(\CC[-1]) \simeq \Perf_{\kk[\beta]}(\cV)$-module structure on $\cC_\Phi$ and $\cC_{\Psi},$ respectively.\footnote{Physicists often work with $\ZZ \times \ZZ/2\ZZ$-graded categories with a super Koszul sign rule. This allows one to define a square root of the  ``shearing'' autoequivalence from \Cref{obs:sigma} that can be used to modify the constructions in this paper so that $x$ and $\beta$ will both have degree $(1,1)$. This ensures that invariants attached to surfaces, such as the $\EE_3$-algebra from \Cref{thm:e3-functions}, become finite dimensional in each bi-degree.
}
\end{remark}

\Cref{cor:ungauging-dim1} becomes more geometrically interesting when we take invariants for the actions of higher-dimensional tori. To do this, we will need to replace $\Sph$ by a more interesting 2-category $\Sph_n.$ (In terms of the geometry described in \Cref{subsec:abelian-intro}: just as
$\Sph$ was associated to the geometry of the space $\CC$ with its toric stratification, the 2-category $\Sph_n$ will be associated to $\CC^n$ with its toric stratification.)
\begin{definition}\label{defn:n-spherical-functors}
We write $\Sph_n$ for the 2-category $(\Sph)^{\otimes n}$ obtained as the $n$-fold tensor product (taken in $\PrLSt_2$) of $\Sph$ with itself. Concretely, this is the 2-category of ``spherical $n$-cubes'': functors $\Adj^{\times n} \ra \St$ such that all 1-morphisms in $\Adj^{\times n}$ are carried to spherical functors. This 2-category has $n$ commuting automorphisms of the identity, coming from the monodromies (as described in \Cref{prop:monodromy-intertwining}) associated to the $n$ spherical adjunctions.
\end{definition}

\begin{notation}
\label{notation.fix.ntorus.D.and.subtorus.G.and.quotient.F}
In what follows, we write $D : =(\CC^\times)^n$ for the $n$-torus, with dual $D^\vee \cong (\CC^\times)^n.$ We will also fix a subtorus $G \subset D$ with quotient $F$, so that we have an exact sequence of tori
\[
\begin{tikzcd}
1 \arrow[r] &
G \arrow[r, "i"]&
D \arrow[r, "p"]&
F \arrow[r] &
1
\end{tikzcd}
\]
with dual exact sequence of tori
\[
\begin{tikzcd}
1 \arrow[r]&
F^\vee \arrow[r, "p^\vee"]&
D^\vee \arrow[r, "i^\vee"]&
G^\vee \arrow[r]&
1~.
\end{tikzcd}
\]
\end{notation}
\begin{proposition}\label{prop:sph-n-calc}
Let $X^n := (\CC\sqcup 0)^n/D^\vee$ and $Y^n := \CC^n/D^\vee.$ Then $\Sph_n$ is equivalent to the 2-category of modules over the monoidal category $\Coh(X^n\times_{Y^n}X^n).$

Moreover, these 2-categories may be lifted to objects of the 3-category $\Loc^{(3)}(BD) \simeq \QCoh^{(3)}(BD^\vee),$ where the first structure is as described in \Cref{defn:n-spherical-functors} and the second comes from the $D^\vee$-linear structure of the category $\Coh(Y^n\times_{X^n} Y^n).$
\end{proposition}
\begin{proof}
Observe the identification $X^n \times_{Y^n} X^n \cong (X^1 \times_{Y^1} X^1)^{\times n}$. Using \Cref{prop:monodromy-intertwining}, we obtain the composite equivalence
\begin{equation}
    \label{eq:product-sphn}
\Coh(X^n \times_{Y^n} X^n)
\simeq
\Coh ( ( X^1 \times_{Y^1} X^1)^{\times n})
\simeq
\Coh ( X^1 \times_{Y^1} X^1 )^{\otimes n},
\end{equation}
and by definition $\Sph_n:=\Sph^{\otimes n}$ is the 2-category of modules over the RHS.
Moreover, the tensor product structure on the RHS of \Cref{eq:product-sphn} corresponds to the tensor product factorization $\Perf(BD^\vee)\simeq \Perf(B\CC^\times)^{\otimes n}$ coming from the product structure $D^\vee = (\CC^\times)^n.$
Using the dual factorizations $D=(\CC^\times)^n$ and $\pi_1 D = (\pi_1\CC^\times)^n = \ZZ^n,$
we conclude that as an object in $\QCoh^{(3)}(BD^\vee)\simeq \Loc^{(3)}(BD),$ the 2-category of modules over $\Cref{eq:product-sphn}$ has underlying monodromy map $\pi_1 D = \ZZ^n \to \End(\id_{\Sph_n})$ selecting the monodromy automorphisms of \Cref{prop:monodromy-intertwining} in each of the $n$ tensor factors.
\end{proof}

As in \Cref{cor:ungauging-dim1}, we may therefore deduce a description of the invariant 2-category $\Sph_n^{D}$ by $D^\vee$-deequivariantization in the spectral description of $\Sph_n$. More generally, we can describe the invariant 2-category $\Sph_n^G$ by $G^\vee$-deequivariantization.

\begin{theorem}\label{thm:higherdim-mainthm}
  Let $X^n_G = (\CC\sqcup 0)^n / F^\vee$ and  $Y^n_G = \CC^n/F^\vee.$ Then $\Sph_n^G$ is equivalent to the 2-category of modules for the monoidal category $\cA_G:=\Coh(X^n_G\times_{Y^n_G} X^n_G).$
\end{theorem}

\begin{proof}
Let $\pi: G \to \{1\}$ be the projection, and consider the iterated pullback squares
\[
\begin{tikzcd}
X^n_G \arrow[r] \arrow[dr, phantom, "\ulcorner", very near start] \arrow[d] 
& Y^n_G \arrow[r] \arrow[dr, phantom, "\ulcorner", very near start] \arrow[d] 
& B F^\vee \arrow[r] \arrow[dr, phantom, "\ulcorner", very near start] \arrow[d, "Bp^\vee"'] 
& B \{1\} \arrow[d, "B\pi^\vee"]
\\
X^n \arrow[r] 
& Y^n \arrow[r] 
& B D^\vee \arrow[r, "Bi^\vee"]
& B G^\vee~.
\end{tikzcd}
\]
We now deduce the equivalences
\begin{align}
\label{gd.eq1} (B \pi)_* (B i)^* \Sph_n &\simeq (B \pi^\vee)^! (B i^\vee)_* \Coh(X^n \times_{Y^n} X^n) \\
\label{gd.eq2} &  \simeq (B \pi^\vee)^* (B i^\vee)_* \Coh(X^n \times_{Y^n} X^n) \\
\nonumber
& \simeq (B \pi^\vee)^* \Coh(X^n \times_{Y^n} X^n) \\
\label{gd.eq4} & \simeq \Coh(X^n_G\times_{Y^n_G} X^n_G)~,
\end{align}
where \Cref{gd.eq1} follows from Propositions \ref{prop:sph-n-calc} and \ref{prop:monodromy-intertwining}, \Cref{gd.eq2} follows from \Cref{lem:bside-ambidexterity}, and \Cref{gd.eq4} follows from \Cref{lem:algebraic-ungauging}.
\end{proof}

\noindent An interpretation of \Cref{thm:higherdim-mainthm} as 3d mirror symmetry between Gale dual toric cotangent stacks is given as \Cref{mainthm:boundary-duality} in \Cref{subsec:abelian-intro}.

Finally, from the 2-category $\Sph_n^G\simeq \Mod_{\cA_G}(\St),$ we may extract an $\EE_3$-algebra:
\begin{theorem}\label{thm:e3-functions}
  Let $\uno_{Z(\cA_G)}$ be the unit object in the center $Z(\cA_G)$ of the monoidal category $\cA_G.$ Then there is an equivalence
  \begin{equation}\label{eq:e3-functions}
  \End_{Z(\cA_G)}(\uno_{Z(\cA_G)}) \simeq \cO(T^*[2]\CC^n/F^\vee)
  \end{equation}
  between the endomorphism algebra of $\uno_{Z(\cA_G)}$ and the algebra of ``functions on the shifted cotangent stack $T^*[2]\CC^n/F^\vee$'': by definition, this is the algebra of functions on $T^*\CC^n/F^\vee,$ with homological grading changed so that functions linear in the cotangent direction are in degree 2.
\end{theorem}
\begin{proof}
By \cite{BZNP}*{Theorem 1.2.10}, the center $Z(\cA_G)$ of the monoidal category $\cA_G$ is equivalent to the category $\Coh_{prop/(\CC^n/F^\vee)}(\cL(\CC^n/F^\vee))$ of coherent sheaves on the loop stack $\CC^n/F^\vee$ that are proper over $\CC^n/F^\vee.$ The unit object in this $\EE_2$-category is the pushforward
\[
\uno_{Z(\cA_G)} = i_*\cO_{\CC^n/F^\vee}
\]
of the structure sheaf of $\CC^n/F^\vee$ along the map $i:\CC^n/F^\vee\to \cL\CC^n/F^\vee$ representing the inclusion of constant loops.

As the pushforward $i_*\cO_{\CC^n/F^\vee}$ lives in the full subcategory of $\Coh_{prop/(\CC^n/F^\vee)}(\cL(\CC^n/F^\vee))$ on the objects set-theoretically supported on the constant loops $\CC^n/F^\vee,$ we may
%
therefore perform our computation inside 
this full subcategory, or equivalently inside the category of coherent sheaves on
the formal neighborhood (in $\cL\CC^n/F^\vee$) of the space of constant loops: this 
formal neighborhood, which we denote
by $\widehat{\cL}\CC^n/F^\vee,$ is by definition the {\em formal loop space} of $\CC^n/F^\vee.$
By \cite{BZN-connections}*{Theorem 1.23}, there is an equivalence
$\widehat{\cL}\CC^n/F^\vee\simeq\widehat{{\rm T}}[-1](\CC^n/F^\vee)$
between the formal loop space of the stack and its shifted tangent space. The algebra to be computed in \Cref{eq:e3-functions} is therefore equivalent to $\End_{\Coh(\widehat{{\rm T}}[-1]\CC^n/F^\vee)}(\cO_{\CC^n/F^\vee}),$ and the result follows by a fiberwise application of the standard Koszul duality equivalence (described for instance in \cite{Lurie-Rotation}*{\S 3.4} or \cite{BZN-connections}*{\S 5.1}) relating coherent sheaves on the $(-1)$-shifted tangent fiber to perfect modules over functions on the ``2-shifted cotangent fiber,'' exchanging the zero-section with the rank-1 free module.
\end{proof}

\appendix

\section{Categorical conventions}
\label{section.appendix.categorical.conventions}

\subsection{Generalities}

Throughout this paper, we adhere to the ``implicit $\infty$'' convention: all definitions and constructions should be interpreted homotopy-coherently, so that for example by ``$n$-category'' we mean ``$(\infty,n)$-category''. We nevertheless occasionally use ``$\infty$'' for emphasis; we use the word ``discrete'' to emphasize that we are referring to a non-homotopical object.

We take Lurie's books \cite{Lurie-HTT,Lurie-HA} as background references for category theory, although we also give specific citations where relevant. 
We will also need to draw from the theory of enriched categories, which is set up in \cite{GH-enr}. A summary of the basic properties of this theory can be found in \cite{MS-K2}*{\S A.1.12}; in this appendix, we will also recall several auxiliary results about enriched categories proven in \cite{MS-K2}*{\S\S A, B}.

\begin{notation}\label{notation.appendix}
\begin{enumerate}
    \item Given a
symmetric monoidal category $\cW$, 
we write $\Cat(\cW)$ for the symmetric monoidal category of $\cW$-enriched categories. 
    \item If $\cW$ is a category of $n$-categories, we will refer to $\Cat(\cW)$ as an $(n+1)$-category. In particular,
    we write $\Cat_n := \Cat(\Cat_{n-1})$ for the $(n+1)$-category of small $n$-categories. (See \cite{MS-K2}*{Definition A.1.13, Notation A.1.15} for the recursive construction of these $(n+1)$-categories.) As a special case, we write $\Spaces := \Cat_0$ for the 1-category of spaces.
    \item  We write $\iota_k$ to refer to the maximal sub-$k$-category; for any $n \geq k$ this assembles as a functor $\iota_{k+1} \Cat_n \xra{\iota_k} \Cat_k$.
    \item We write $\Alg(\cW) \xra{\fB} \Cat(\cW)$ for the functor taking an algebra object in $\cW$ to the corresponding one-object $\cW$-enriched category.
\end{enumerate}
\end{notation}

Recall that among monoidal categories one can contemplate both strictly and laxly monoidal functors. We note that the right adjoint of a strictly monoidal functor is canonically laxly monoidal. We note too that a laxly monoidal functor can satisfy the condition of being strictly unital. Identical remarks apply in the symmetric monoidal case. Moreover, given an operad $\cO$, an adjunction in which the left adjoint is symmetric monoidal lifts to an adjunction on $\cO$-algebras. (See for instance \cite{GH-enr}*{Proposition A.5.11} for details on the above.)

We use the device of Grothendieck universes, employing the terms ``small'', ``large'', and ``huge'' accordingly. These terms should always be interpreted inclusively, in the sense that e.g.\! the term ``large space'' is shorthand for ``a space that is either small or large''.  In all cases of interest to us, every large (resp.\! huge) category is locally small (resp.\! locally large); this applies both to enriched and unenriched categories. We use hats to denote passage to a larger context, e.g.\! we write $\what{\Spaces}$ to denote the huge category of large spaces. We mostly suppress discussion of size, except where we find it clarifying.

\subsection{Stable $n$-categories}\label{subsec:stable-ncats}

We work $\kk$-linearly, for $\kk$ a fixed commutative ring or commutative ring spectrum. 

We write $\Stu_\SS \subset \Cat$ for the 1-full sub-2-category of small stable categories. We consider this as equipped with the symmetric monoidal structure that corepresents multiexact functors.

We write $\cV := \Perf_\kk \in \CAlg(\Stu_\SS)$ for the stably symmetric monoidal category of perfect $\kk$-modules. For the present discussion it will also be useful to write $\what{\cV} := \Mod_\kk \simeq \Ind(\cV)$ for the presentably symmetric monoidal stable category of $\kk$-modules.

We write 
$\Stu_\kk \simeq \Mod_\cV(\Stu_\SS)$ 
for the 2-category of small stable $\kk$-linear categories. We simply write $\otimes := \otimes_\cV$ for its symmetric monoidal structure, and $\Fun^\ex$ for its adjoint self-enrichment.

We denote by
\[
\begin{tikzcd}[column sep=2cm]
\Spaces
\arrow[yshift=0.9ex]{r}{\Sigma^{(\infty,0)}_+}
\arrow[leftarrow, yshift=-0.9ex]{r}[yshift=-0.2ex]{\bot}[swap]{\Omega^{(\infty,0)}}
&
\what{\cV}
\end{tikzcd}
\quad
\text{and}
\quad
\begin{tikzcd}[column sep=2cm]
\Cat
\arrow[yshift=0.9ex]{r}{\Sigma^{(\infty,1)}_+}
\arrow[leftarrow, yshift=-0.9ex]{r}[yshift=-0.2ex]{\bot}[swap]{\fgt}
&
\St
\end{tikzcd}
\]
the indicated free/forget adjoints.\footnote{The left adjoint $\Sigma^{(\infty,0)}_+$ may be referred to as the ``$\kk$-linear chains'' functor.} Both left adjoints are symmetric monoidal.
The second adjunction is obtained by applying $\Cat(-)$ to the first and then composing with the adjunction
\[
\begin{tikzcd}[column sep=2cm]
\Cat(\what{\cV})
\arrow[yshift=0.9ex]{r}{\Env}
\arrow[hookleftarrow, yshift=-0.9ex]{r}[yshift=-0.2ex]{\bot}
&
\St
\end{tikzcd}
~,
\]
whose (symmetric monoidal) left adjoint is given by the stable envelope and whose right adjoint is fully faithful. (See \cite{MS-K2}*{\S B.1} for more details on the stable envelope and its fully faithful adjoint.)

Given a category $\cA \in \Cat$, we note that $\Sigma^{(\infty,1)}_+ \cA \in \St$ can be characterized as the \textit{finite $\kk$-linear presheaves} on $\cA$: the smallest stable subcategory containing the image of the stabilized Yoneda functor
\[
\cA
\longhookra
\Fun(\cA^\op,\Spaces)
\xra{\Sigma^{(\infty,0)}_+}
\Fun(\cA^\op,\what{\cV})
~.
\]
In these terms, the symmetric monoidality of $\Sigma^{(\infty,1)}_+$ is incarnated via Day convolution.

For any $n \geq 2$, we inductively define a \bit{stable $n$-category} to be a category enriched in the symmetric monoidal $n$-category $\Stu_{\kk,n-1}$ of stable $(n-1)$-categories; these assemble into a (self-enriched) $(n+1)$-category $\Stu_{\kk,n} := \Cat(\Stu_{\kk,n-1})$;\footnote{In the case that $n=2$, this does not quite match the conventions of \cite{MS-K2}, which additionally require the existence of finite sums.}
we write $\Fun^{n-\ex}$ for the self-enrichment.
(We will only need these definitions in the cases $n=2,3$.) 
We inductively obtain an adjunction
\[
\begin{tikzcd}[column sep=2cm]
\Cat_n
\arrow[yshift=0.9ex]{r}{\Sigma^{(\infty,n)}_+}
\arrow[leftarrow, yshift=-0.9ex]{r}[yshift=-0.2ex]{\bot}[swap]{\fgt}
&
\Stu_{\kk,n}
\end{tikzcd}
\]
by applying $\Cat(-)$, whose left adjoint is symmetric monoidal. We generally omit the right adjoint $\fgt$ from our notation.\footnote{This forgetful functor is conservative, and is moreover faithful in the case that $\kk = \SS$ is the sphere spectrum (or more generally whenever $\kk$ is idempotent).}

We note here that stable 2-categories are ``pre-semiadditive'': inside of any stable 2-category, any initial object is also terminal and reversely, and thereafter any finite coproduct is also a finite product and reversely.

\subsection{Presentability}

We write $\Pr^L \subset \what{\Cat}$ for the 1-full sub-2-category whose objects are presentable categories and whose morphisms are left adjoint functors. We consider this as equipped with the symmetric monoidal structure that corepresents multicocontinuous functors. We write $\Fun^L$ for its adjoint self-enrichment.

We employ the theory of presentable enriched categories of \cite[\sec A.3]{MS-K2}. (Although see also \cite{Stefanich-Pres} for a more general theory of presentable $n$-categories.) Our usage thereof may be summarized as follows.

Given a presentably symmetric monoidal category $\cW \in \CAlg(\Pr^L)$, we write $\iota_1 \Pr^L_\cW \subseteq \what{\Cat}(\cW)$ for the huge category of \bit{presentable $\cW$-enriched categories}. 
This category is defined in \cite{MS-K2}*{Definition A.3.2}, but we will work with this category via \cite{MS-K2}*{Theorem A.3.8}, which gives a description of this category
\begin{equation}\label{eq:pres-symmon-pres}
\iota_1 \Pr^L_\cW
\xlongra{\sim}
\Mod_\cW(\iota_1 \Pr^L)
~
\end{equation}
as the category of presentable categories with $\cW$-action.
The equivalence \eqref{eq:pres-symmon-pres} carries a presentable $\cW$-enriched category $\cC \in \iota_1 \Pr^L_\cW$ to its underlying (unenriched) presentable category $U(\cC) \in \iota_1 \Pr^L$ equipped with the $\cW$-action given by tensoring, which is characterized by the universal property that
\[
\hom_{U(\cC)}(W \cdot C , D)
\simeq
\hom_\cW ( W , \hom_\cC ( C , D ) )
\]
for all $W \in \cW$ and all $C,D \in \cC$. We write $\Fun^L_\cW$ for the self-enrichment of $\iota_1 \Pr^L_\cW \simeq \Mod_\cW(\iota_1 \Pr^L)$, which is adjoint to its symmetric monoidal structure $\otimes_\cW$ (computed in $\Mod_\cW(\iota_1 \Pr^L)$).

As a special case, we write $\iota_1 \Pr^L_n := \iota_1 \Pr^L_{\iota_1 \Cat_{n-1}}$, and refer to its objects as \bit{presentable $n$-categories}. Similarly, we write
\[
\iota_1 \Pr^{L,\st_{\kk}}
:=
\iota_1 \Pr^{L,\st_{\kk}}_1
:=
\iota_1 \Pr^L_{\what{\cV}}
\quad
\text{and}
\quad
\iota_1 \Pr^{L,\st_{\kk}}_2
:=
\iota_1 \Pr^L_\St
\]
and respectively refer to objects of these as \bit{presentable stable 1-} or \bit{2-categories}.

Given a morphism $\cW \ra \cW'$ in $\CAlg(\Pr^L)$, we obtain adjoint functors
\[ \begin{tikzcd}[column sep=1.5cm]
\iota_1 \Pr^L_\cW
\simeq
&[-1.7cm]
\Mod_\cW(\iota_1 \Pr^L)
\arrow[bend left]{r}{\cW' \otimes_\cW (-)}
\arrow[leftarrow]{r}[description]{\fgt}[yshift=1.5ex]{\bot}[swap, yshift=-1.5ex]{\bot}
\arrow[bend right]{r}[swap]{\Fun^L_\cW(\cW',-)}
&
\Mod_{\cW'}(\iota_1 \Pr^L)
&[-1.7cm]
\simeq
\iota_1 \Pr^L_{\cW'}
\end{tikzcd}
~.
\]
In particular, it follows that the forgetful functor $\iota_1 \Pr^L_\cW \xla{\fgt} \iota_1 \Pr^L_{\cW'}$ commutes with both limits and colimits.

\subsection{Enriching adjoint functors}
In this section, we lift adjunctions of unenriched categories to enriched adjunctions, using the fact that an unenriched adjoint commuting with (co)tensors has an enriched adjoint (\cite{MS-K2}*{Lemma A.2.14}). Combining this fact with the unenriched adjoint functor theorem for presentable categories, we obtain the following:

\begin{theorem}[\cite{MS-K2}*{Corollary A.3.6}]\label{thm:enriched-adjt-functors}
Let $\cC,\cD\in \iota_1\Pr_\cW^L$ be presentable $\cW$-categories.
\begin{enumerate}
    \item A functor $\cC\xra{F}\cD$ is a left adjoint if and only if it preserves colimits and tensors.
    \item A functor $\cC\xla{G} \cD$ is a right adjoint if and only if it preserves limits, cotensors, and $\kappa$-filtered colimits for some regular cardinal $\kappa.$
    \qed
\end{enumerate}
\end{theorem}

The hypotheses of this theorem will be satisfied for certain 2-categories of functors we consider.

\begin{lemma}\label{lem:functorcats.presentability}
Let $\cC$ be a small 2-category and consider the 2-category $\Fun(\cC, \St).$ The underlying 1-category $\iota_1\Fun(\cC,\St)$ is presentable, with limits and colimits computed pointwise, and the 2-category $\Fun(\cC,\St)$ is tensored and cotensored over $\St,$ with tensors and cotensors computed pointwise.
\end{lemma}
\begin{proof}
This is essentially the content of \cite{MS-K2}*{Propositions A.3.7, B.1.8}. In more detail: \cite{Lurie-HTT}*{Proposition A.3.7.6} establishes that presentability for an $\infty$-category $\cC$ is equivalent to the existence of a combinatorial simplicial model category presenting it. In the case of $\Fun(\cC,\St),$ such a combinatorial presentation is available in \cite{GHL}*{Proposition 3.3.1}. The statements about co/tensors and co/limits can be checked directly in this model.
\end{proof}

\Cref{lem:functorcats.presentability} will be used in the following section in our analysis of the 2-category $\Fun(\Adj,\St)$ of stable adjunctions. In the case of functor 2-categories in which the domain is a 1-category, the situation is even simpler. The following construction is used in \Cref{defn.Loc.three}:
\begin{lemma}\label{lem:loc3.adjoints}
Let $f:G\to H$ be a homomorphism of tori. Then the functor
\[
(B f) ^* : \Loc^{(3)}(BG) \to \Loc^{(3)}(BH)
\]
has both adjoints.
\end{lemma}
\begin{proof}
Note that since $BG_{\sf B}$ is a 1-category (in fact a 0-category), $BG_{\sf B} = \iota_1 BG_{\sf B},$ and therefore the underlying 1-category of $\Loc^{(3)}(BG) = \Fun(BG_{\sf B}, \Pr^{L,\st_{\kk}}_2)$ is given by the 1-categorical functor category
$\Fun(\iota_1 BG_{\sf B}, \iota_1\Pr^{L,\st_{\kk}}_2).$ The unenriched functor $\iota_1 (Bf)^*$ admits both left and right adjoints given by left and right Kan extension (note that $\iota_1 \Pr^{L,\st_{\kk}}_2$ admits all limits and colimits). Furthermore, since (co)limits and (co)tensors in a functor category are computed pointwise, these can be upgraded to adjoints of $(Bf)^*.$
\end{proof}

\section{Monads and adjunctions}
In this appendix we recall some features of the theory of adjunctions in a 2-category $\cC$. This theory is stated in terms of the {\em free adjunction}, a 2-category $\Adj$ described in \cite{RV-adj} (following an earlier description in \cite{Schanuel-Street} in the setting of discrete categories) with the universal property (\cite{RV-adj}*{Proposition 3.3.4}) that adjunctions in $\cC$ correspond to functors $\Adj\to \cC.$ 
One application of this theory is the following result, which justifies use of the phrases ``spherical functor'' and ``spherical adjunction'' interchangeably in our notation.
\begin{theorem}[\cite{Haugseng-adjunctions}*{Theorem 1.1}, extending \cite{RV-adj}*{Theorem 4.4.18}\footnote{\cite{RV-adj} gives an equivalence of the underlying $\infty$-groupoids, whereas \cite{Haugseng-adjunctions} establishes an equivalence of 2-categories.}]\label{thm:adjunctions.are.adjoints}
Let $\cC$ be a 2-category. Then there is an equivalence $\Fun(\Adj,\cC)\simeq \Fun([1],\cC)_{l.adj.}$ between adjunctions in $\cC$ and left adjoint 1-morphisms in $\cC.$
\qed
\end{theorem}

We now collect some notation and basic facts which we will use to describe the free adjunction.
\begin{notation}
\begin{enumerate}
    \item We write $\ell,r$ for the two objects of the 2-category $\Adj.$
    \item We write $\ell\xra{L} r$ and $\ell\xla{R} r,$ respectively, for the universal left and right adjoint 1-morphisms between these. 
    \item We write $\Mnd\subset \Adj$ for the full subcategory on the object $\ell\in \Adj.$ As the notation suggests, $\Mnd\simeq\fB\bDelta_+$ is the {\em free monad} \cite{RV-adj}*{\S 6}.
    \item For $\cC$ a 2-category, we write $\Fun(\Mnd,\cC)\xla{\ev_\ell}\Fun(\Adj,\cC)$ for the restriction along the inclusion $\Mnd\hookrightarrow \Adj,$ so that $\ev_\ell$ associates to an adjunction in $\cC$ its underlying monad.
    \item We denote the respective left and right adjoints of $\ev_\ell$ (constructed below in \Cref{lem:KL-EM}) by $\Kl$ and $\EM.$
\end{enumerate}
\end{notation}

%
%
%

We will be interested in the case where $\cC=\St$ is the 2-category of small stable categories. 
%
\begin{lemma}\label{lem:KL-EM}
The restriction functor $\Fun(\Mnd,\St)\xla{\ev_\ell} \Fun(\Adj,\St)$ has both adjoints. 
\end{lemma}
\begin{proof}
From \Cref{lem:functorcats.presentability}, the functor 2-categories $\Fun(\Mnd,\St)$ and $\Fun(\Adj,\St)$ are presentable and the restriction $\ev_\ell$ commutes with co/limits and co/tensors. Therefore, we may apply \Cref{thm:enriched-adjt-functors} to ensure the existence of left and right adjoints.
\end{proof}
The adjoints to $\ev_\ell$ are functors $\Fun(\Mnd,\St)\to \Fun(\Adj,\St)$ that produce an adjunction in $\St$ from a monad in $\St$: as the notation suggests, these are stable $\infty$-categorical versions of the Kleisli and Eilenberg--Moore constructions, respectively. More details on these constructions can be found in \cite{RV-adj}*{\S 6} and \cite{Christ-spherical}*{\S 1.4} but we will only need their universal properties and the following recognition principle for Kleisli adjunctions.

\begin{proposition}\label{prop:kl-of-surj}
An adjunction $(L\dashv R)\in \Fun(\Adj,\St)$ is Kleisli (i.e., it may be recovered as the Kleisli adjunction $\Kl(R\circ L)$ associated to its monad $R\circ L$) if and and only if $L$ is essentially surjective up to stablization.\footnote{The essential image of an
an exact functor $F: \cC \to \cD$ between stable categories is not necessarily stable. We say that $F$ is essentially surjective up to stabilization if $\cD \cong \Sigma^{(\infty,1)}_+ \Image F$.} 
\end{proposition}
\begin{remark}
    In the setting of discrete category theory, \Cref{prop:kl-of-surj} is essentially a tautology.\footnote{We learned this fact from Peter LeFanu Lumsdaine's answer at \url{https://mathoverflow.net/questions/26075/characterization-of-kleisli-adjunctions}.} This is because the Kleisli construction of a monad on a discrete category $\cC$ is usually {\em defined} to be a discrete category with the same object set as $\cC$. Since our Kleisli construction is defined by appealing to an abstract functor theorem, we can no longer refer to the set of objects in a specific model, but the result can just as well be deduced from the universal property of $\Kl.$
\end{remark}
\begin{proof}[Proof of \Cref{prop:kl-of-surj}]
Let $M$ be an exact monad on a small stable category $\cC$ and let $\Kl(M) =: (L_{\Kl(M)} \dashv R_{\Kl(M)})$ be the Kleisli adjunction. We will show the left adjoint $L_{\Kl(M)}$ is essentially surjective up to stabilization. Consider the essential image ${\sf Im \,}L_{\Kl(M)}$ of $L_{\Kl(M)}$. There is a morphism of adjunctions
\[
\begin{tikzcd}[column sep=2cm]
\cC
\arrow[yshift=0.9ex]{r}{L_{\Kl(M)}}
\arrow[leftarrow, yshift=-0.9ex]{r}[yshift=-0.2ex]{\bot}[swap]{R_{\Kl(M)} \circ I} \arrow[d, "\id_\cC"']
&
\Sigma^{(\infty,1)}_+ \Image L_{\Kl(M)} \arrow[d, hookrightarrow ,"I"] \\
\cC \arrow[yshift=0.9ex]{r}{L_{\Kl(M)}}
\arrow[leftarrow, yshift=-0.9ex]{r}[yshift=-0.2ex]{\bot}[swap]{R_{\Kl(M)}}
&
\Kl(M)~,
\end{tikzcd}
\]
where $I$ is the full and faithful inclusion. By the universal property of the Kleisli adjunction, the morphism $(\id_{\cC}, I)$ has a right inverse. Thus $I$ is also essentially surjective and hence an equivalence.

Let $(L\dashv R)\in \Fun(\Adj,\St)$ be an adjunction of stable categories, and let $M := R \circ L$ be the corresponding monad. By the universal property of the Kleisli adjunction, we have a morphism of adjunctions
\[
\begin{tikzcd}[column sep=2cm]
\cC \arrow[yshift=0.9ex]{r}{L_{\Kl(M)}}
\arrow[leftarrow, yshift=-0.9ex]{r}[yshift=-0.2ex]{\bot}[swap]{R_{\Kl(M)}} \arrow[d, "\id_\cC"']
&
\Kl(M) \arrow[d, "F"] \\
\cC \arrow[yshift=0.9ex]{r}{L}
\arrow[leftarrow, yshift=-0.9ex]{r}[yshift=-0.2ex]{\bot}[swap]{R}
&
\cD ~.
\end{tikzcd}
\]
Notice that $F$ is an equivalence if and only if $L$ is essentially surjective up to stablization and $F$ restricts to an equivalence between $\Image L_{\Kl(M)}$ and $\Image L$. Since $L \simeq F \circ L_{\Kl(M)}$ essential surjectivity of $F|_{\Image L_{\Kl(M)}}$ is clear. To see that $F|_{\Image L_{\Kl(M)}}$ is fully faithful, it is enough to notice the action of $F$ on morphism spaces is given by the canonical equivalence
\[
\hom_{\Kl(M)}(L_{\Kl(M)}(c), L_{\Kl(M)}(c')) \simeq \hom_{\cC}(c, M(c')) \simeq \hom_{\cD}(L(c), L(c'))~,
\]
which follows from the fact that $(\id_\cC,F)$ is a morphism of adjunctions.
\end{proof}

\section{Monoidal localizations of symmetric monoidal categories}
The purpose of this section is to prove the following result, which is used in the proof of \Cref{thm.gr.coreps.LHS}.

\begin{lemma}\label{prop:BN-localization}
The inclusion $\ZZ_{\leq 0} \hookra \ZZ$ is the monoidal localization at the object $-1$.
\end{lemma}

\noindent We will deduce \Cref{prop:BN-localization} from the more general fact that symmetric monoidal localizations compute monoidal localizations (\Cref{prop.smloc.computes.mloc}).

We first prove the symmetric monoidal version of \Cref{prop:BN-localization}.

\begin{lemma}
\label{lem.sm.loc.of.Zleqzero}
The inclusion $\ZZ_{\leq 0} \hookra \ZZ$ is the symmetric monoidal localization at the object $-1$.
\end{lemma}

\begin{proof}
By \cite{Robalo}*{Corollary 2.22}, the symmetric monoidal localization of $\ZZ_{\leq 0}$ is given by
\begin{equation}
\label{telescope.for.sm.loczn.of.Zleqzero}
\colim \left(
\ZZ_{\leq 0}
\xra{+ (-1)}
\ZZ_{\leq 0}
\xra{+ (-1)}
\cdots
\right)
~.
\end{equation}
The symmetric monoidal functor $\Cref{telescope.for.sm.loczn.of.Zleqzero} \ra \ZZ$ arising from its universal property acts on the $n^{\rm th}$ term as the functor $\ZZ_{\leq 0} \xra{+n} \ZZ$. This is clearly surjective, and it is fully faithful since the extraction of hom-spaces (being a finite limit) commutes with filtered colimits.
\end{proof}

In order to prove that these localizations agree, we pass to a more general framework.

\begin{definition}
A \bit{marking} of a monoidal or symmetric monoidal category is a set of equivalence classes of objects that contains all the invertible objects. Similarly, a \bit{marking} of a 2-category is a set of equivalence classes of 1-morphisms that contains all the equivalences.\footnote{These are useful in the proof of \Cref{lem.monloc.pres.finite.products}, since $\Mon(\Cat)$ is not cartesian closed.} We write $\CMon(\Cat)'$, $\Mon(\Cat)'$, and $\Cat_2'$ for the evident categories of these.
\end{definition}

\begin{observation}
We have a solid commutative diagram
\[
\begin{tikzcd}[column sep=1.5cm]
\CMon(\Cat)'
\arrow[yshift=0.9ex, dashed]{r}{\tilde{L}}
\arrow[hookleftarrow, yshift=-0.9ex]{r}[swap]{\tilde{\min}}[yshift=-0.2ex]{\bot}
\arrow{d}[swap]{\fgt'}
&
\CMon(\Cat)
\arrow{d}{\fgt}
\\
\Mon(\Cat)'
\arrow[yshift=0.9ex, dashed]{r}{L}
\arrow[hookleftarrow, yshift=-0.9ex]{r}[swap]{\min}[yshift=-0.2ex]{\bot}
&
\Mon(\Cat)
\\
\Cat_{2,*}'
\arrow[yshift=0.9ex, dashed]{r}{L_2}
\arrow[hookleftarrow, yshift=-0.9ex]{r}[swap]{\min_2}[yshift=-0.2ex]{\bot}
\arrow[hookleftarrow]{u}{\fB'}
&
\Cat_{2,*}
\arrow[hookleftarrow]{u}[swap]{\fB}
\end{tikzcd}
~,
\]
in which the horizontal functors are the fully faithful inclusions given by minimal markings.\footnote{Here, $\Cat_{2,*} := (\Cat_2)_{\pt/}$ denotes the category of pointed 2-categories, and $\fB$ denotes the fully faithful delooping functor (with $\fB'$ given by marking the 1-morphisms that correspond to the marked objects).} Moreover, the minimal marking functors preserve limits, and hence by presentability they admit the indicated left adjoint localization functors.
\end{observation}

Here is our main result.

\begin{proposition}
\label{prop.smloc.computes.mloc}
The diagram
\[ \begin{tikzcd}[column sep=1.5cm]
\CMon(\Cat)'
\arrow{r}{\tilde{L}}
\arrow{d}[swap]{\fgt'}
&
\CMon(\Cat)
\arrow{d}{\fgt}
\\
\Mon(\Cat)'
\arrow{r}[swap]{L}
&
\Mon(\Cat)
\end{tikzcd} \]
commutes (i.e.\! the canonical morphism $L \circ \fgt' \ra \fgt \circ \tilde{L}$ is an equivalence).
\end{proposition}

Of course, this immediately implies our original claim.

\begin{proof}[Proof of \Cref{prop:BN-localization}]
This follows by combining \Cref{lem.sm.loc.of.Zleqzero} and \Cref{prop.smloc.computes.mloc}.
\end{proof}

Our proof of \Cref{prop.smloc.computes.mloc} uses the following result.

\begin{lemma}
\label{lem.monloc.pres.finite.products}
The localization functor $\Mon(\Cat)' \xra{L} \Mon(\Cat)$ preserves finite products.
\end{lemma}

\begin{proof}
Comparing universal properties, we see that the canonical morphism $L_2 \circ \fB' \ra \fB \circ L$ is an equivalence. It is clear that the fully faithful inclusions $\fB'$ and $\fB$ preserve finite products, so it suffices to show that $L_2$ preserves finite products. By \cite{MG-rezk}*{Lemma 1.22}, this follows from the fact that $\min_2$ is the inclusion of an exponential ideal, which in turn follows from the fact that a natural transformation is an equivalence if and only if all of its components are equivalences (see e.g.\! \cite[Lemma 2.5.1]{Macpherson-bivariant}).
\end{proof}

\begin{proof}[Proof of \Cref{prop.smloc.computes.mloc}]
It follows from \Cref{lem.monloc.pres.finite.products} that the adjunction $L \adj \min$ is cartesian symmetric monoidal. Hence, it induces an adjunction
\[ \begin{tikzcd}[column sep=2cm]
\CMon(\Mon(\Cat)')
\arrow[yshift=0.9ex]{r}{\CMon(L)}
\arrow[hookleftarrow, yshift=-0.9ex]{r}[swap]{\CMon(\min)}[yshift=-0.2ex]{\bot}
&
\CMon(\Mon(\Cat))
\end{tikzcd} \]
that commutes with the forgetful functors to the adjunction $L \adj \min$. Moreover, by by Dunn additivity the adjunction $\CMon(L) \adj \CMon(\min)$ coincides with the adjunction $\tilde{L} \adj \tilde{\min}$.
\end{proof}

\bibliographystyle{plain}
\bibliography{refs}
\end{document}